\documentclass[11pt,reqno]{amsart}
\usepackage{amsfonts,amsmath,amssymb,amscd,amsthm}
\usepackage{hyperref}
\usepackage[usenames,dvipsnames]{xcolor}
\usepackage{graphicx}
\usepackage{enumerate}
\usepackage{enumitem}
\usepackage[all]{xy}
\usepackage{mathtools}
\usepackage{mathrsfs}
\usepackage{xparse}
\usepackage{booktabs}
\usepackage{caption}

\usepackage{textcase}

\usepackage[T1]{fontenc}

\usepackage{tikz}
\usetikzlibrary{arrows,backgrounds}

\definecolor{darkgreen}{rgb}{0.0, 0.26, 0.15}
\definecolor{darkred}{rgb}{0.65,0.15,0}

\newcommand{\PaperTitle}{Eulerianity of Fourier coefficients of automorphic forms}
\newcommand{\DisplayedTitle}{Eulerianity of Fourier coefficients of \protect\\automorphic forms}

\hypersetup{%
    pdftitle={\PaperTitle},         %
    pdfauthor={Dmitry Gourevitch, Henrik P. A. Gustafsson, Axel Kleinschmidt, Daniel Persson, and Siddhartha Sahi},      %
    linktocpage=true,           
    colorlinks=true,            
    linkcolor=blue!75!black,    
    citecolor=green!50!black,   
    filecolor=magenta,          
    urlcolor=blue!75!black      
}

\let\oldtocsubsection=\tocsubsection
\let\oldtocsubsubsection=\tocsubsubsection
\renewcommand{\tocsubsection}[2]{\hspace{1em}\oldtocsubsection{#1}{#2}}
\renewcommand{\tocsubsubsection}[2]{\hspace{2em}\oldtocsubsubsection{#1}{#2}}

\usepackage{etoolbox}
\makeatletter
\patchcmd{\@startsection}
  {\@afterindenttrue}
  {\@afterindentfalse}
  {}{}
\makeatother

\newtheorem{maintheorem}{Theorem}

\newtheorem{maincorollary}[maintheorem]{Corollary}

\theoremstyle{plain}

\newtheorem*{theorem*}{Theorem}
\newtheorem{lemma}{Lemma}[section]
\newtheorem{proposition}[lemma]{Proposition}
\newtheorem{theorem}[lemma]{Theorem}

\newtheorem*{conjecture*}{Conjecture}

\theoremstyle{definition}
\newtheorem{definition}[lemma]{Definition}

\newtheorem*{example*}{Example}

\theoremstyle{remark}

\newtheorem{remark}[lemma]{Remark}
\newtheorem*{remark*}{Remark}

 \theoremstyle{plain}
\newcommand{\iso}{\cong}

\newcommand{\WO}{\operatorname{WO}}

\newcommand{\bs}{\backslash}

\newcommand{\Hom}{\operatorname{Hom}}

\newcommand{\diag}{\operatorname{diag}}

\newcommand{\A}{\mathbb{A}}

\DeclareMathAlphabet{\mathbbold}{U}{bbold}{m}{n}
\newcommand{\id}{\mathbbold{1}}

\newcommand{\Q}{{\mathbb Q}}

\newcommand{\C}{{\mathbb C}}
\newcommand{\K}{{\mathbb K}}
\newcommand{\E}{{\operatorname{E}}}

\newcommand{\Exp}{\operatorname{Exp}}

\newcommand{\alp}{{\alpha}}

\newcommand{\bfG}{{\mathbf{G}}}

\newcommand{\cA}{{\mathcal{A}}}

\newcommand{\cF}{{\mathcal{F}}}

\newcommand{\abs}[1]{\left|{#1}\right|}

\newcommand{\fg}{{\mathfrak{g}}}

\newcommand{\fn}{{\mathfrak{n}}}

\newcommand{\fu}{{\mathfrak{u}}}
\newcommand{\fv}{{\mathfrak{v}}}

\newcommand{\cO}{{\mathcal{O}}}
\newcommand{\GL}{\operatorname{GL}}
\newcommand{\SL}{\operatorname{SL}}
\newcommand{\SO}{\operatorname{SO}}

\newcommand{\ad}{\operatorname{ad}}
\newcommand{\Ad}{\operatorname{Ad}}

\newcommand{\fq}{\mathfrak{q}}
\newcommand{\fr}{\mathfrak{r}}

\newcommand{\into}{{\hookrightarrow}}

\newcommand{\reals}{\mathbb{R}}

\newcommand{\WS}{\operatorname{WS}}

\renewcommand{\sl}{\mathfrak{sl}}

\newcommand{\lie}[1]{\mathfrak{#1}}

\NewDocumentCommand{\intl}{g}{
    \IfNoValueTF{#1}{\int\limits}{\int\limits_{\mathclap{#1}}}
}
\NewDocumentCommand{\suml}{g}{
    \IfNoValueTF{#1}{\sum\limits}{\sum\limits_{\mathclap{#1}}}
}

\newcommand{\Oh}{\mathcal{O}}
\newcommand{\op}{\operatorname}

\numberwithin{equation}{section}

\begin{document}

\author[D. Gourevitch]{Dmitry Gourevitch}
\address{Dmitry Gourevitch,
Faculty of Mathematics and Computer Science,
Weizmann Institute of Science,
POB 26, Rehovot 76100, Israel }
\email{dmitry.gourevitch@weizmann.ac.il}
\urladdr{\url{http://www.wisdom.weizmann.ac.il/~dimagur}}

\author[H. Gustafsson]{Henrik P. A. Gustafsson}
\address{Henrik Gustafsson, \newline School of Mathematics, Institute for Advanced Study, Princeton, NJ~08540 \newline
Department of Mathematics, Rutgers University, Piscataway, NJ~08854 \newline Department of Mathematical Sciences, University of Gothenburg and Chalmers University of Technology, SE-412~96 Gothenburg, Sweden}
\email{gustafsson@ias.edu}
\urladdr{\url{http://hgustafsson.se}}

\author[A. Kleinschmidt]{Axel Kleinschmidt}
\address{Axel Kleinschmidt,  \newline
Max-Planck-Institut f\"{u}r Gravitationsphysik (Albert-Einstein-Institut),
Am M\"{u}hlenberg 1, DE-14476 Potsdam, Germany \newline
 International Solvay Institutes,
ULB-Campus Plaine CP231, BE-1050, Brussels, Belgium}
\email{axel.kleinschmidt@aei.mpg.de}

\author[D. Persson]{Daniel Persson}
\address{Daniel Persson, Chalmers University of Technology, Department of Mathematical Sciences\\
SE-412\,96 Gothenburg, Sweden}
\email{daniel.persson@chalmers.se}

\author[S. Sahi]{Siddhartha Sahi}
\address{Siddhartha Sahi, Department of Mathematics, Rutgers University, Hill Center -
Busch Campus, 110 Frelinghuysen Road Piscataway, NJ 08854-8019, USA}
\email{sahi@math.rugers.edu}



\title[\PaperTitle]{\DisplayedTitle}
\maketitle
\begin{abstract}
We study the question of Eulerianity (factorizability) for Fourier coefficients of automorphic forms, and we prove a general transfer theorem that allows one to deduce the Eulerianity of certain coefficients from that of another coefficient. We also establish a `hidden' invariance property of Fourier coefficients. We apply these results to minimal and next-to-minimal automorphic representations, and deduce Eulerianity for a large class of Fourier and Fourier--Jacobi coefficients.
In particular, we prove Eulerianity for parabolic Fourier coefficients with characters of maximal rank for a class of Eisenstein series in minimal and next-to-minimal representations of groups of ADE-type that are of interest in string theory.

\end{abstract}

\setcounter{tocdepth}{2}
\tableofcontents

\newpage

\section{Introduction}
\label{sec:intro}

In the classical theory of modular forms for $\SL_2$ the study of Fourier coefficients is instrumental for revealing arithmetic information. For example, Fourier coefficients of Eisenstein series carry information on the analytic continuation of zeta functions, while Fourier coefficients of theta functions yield important results for Euclidean lattices (see, e.g., \cite{Serre}). For automorphic forms on groups of higher rank the study of Fourier coefficients is crucial when proving the meromorphic continuation of Langlands $L$-functions as well as analysing the trace formula, both of which are cornerstones in the Langlands program (see, e.g., \cite{MR0419366,MR498494,MR610479,JiangRallis,GinzburgTowards}).

A key step in proving continuation of $L$-functions is the unfolding of global (adelic) integrals, which in particular involves Fourier coefficients of Eisenstein series. It is then important to consider special types of Fourier coefficients which are \emph{Eulerian}, i.e.\ can be factorized into an Euler product over all places. The prototypes of such coefficients are the so-called \emph{Whittaker--Fourier coefficients}, or Whittaker coefficients for short.
For now, let $\mathbf{G}$ be a reductive algebraic group defined and quasi-split over a number field $\mathbb{K}$.
Let $\mathbb{A} = \mathbb{A}_\mathbb{K}$ be the adele ring of $\mathbb{K}$, $G = \mathbf{G}(\mathbb{A})$ and $\Gamma = \mathbf{G}(\mathbb{K})$.
For some of our results we will later also consider non-split groups and finite central extensions of $\mathbf{G}(\mathbb{A})$.
Choose a Borel subgroup $\mathbf{B}=\mathbf{N}\mathbf{A} \subset \mathbf{G}$, with maximal torus $\mathbf{A}$ and unipotent radical $\mathbf{N}$. Let $\psi_N$ be a unitary character on $N = \mathbf{N}(\mathbb{A})$, trivial on $N \cap \Gamma = \mathbf{N}(\mathbb{K})$. Furthermore, let $\eta$ be an automorphic form in an admissible automorphic representation $\pi$ of $G$. To this data one can associate the Whittaker coefficient:
\begin{equation}
  \mathcal{W}_{\psi_N}[\eta](g) =\int_{\mathbf{N}(\mathbb{K})\backslash \mathbf{N}(\mathbb{A})} \eta(ng)\psi_N(n)^{-1} dn \qquad g \in G.
\label{WhittakerIntegral}
\end{equation}

Choose a pinning of $\mathbf{G}$, i.e., for every simple root $\alpha$ a non-zero root vector $X_\alpha\in \fg_{\alpha}$, and let $x_\alpha : \mathbb{G}_a \to \mathbf{G}$ be the corresponding one-parameter subgroup of $\mathbf{G}$.
When $\psi_N$ is generic, that is, $\psi_{N}(x_\alpha(t))$ is non-trivial for all simple roots $\alpha$, it is well-known that $\mathcal{W}_{\psi_{N}}[\eta](g)$ is Eulerian, that is, for a pure tensor $\eta$ in an irreducible admissible representation $\pi = \bigotimes_\nu \pi_\nu$, it factorizes over the places $\nu$ of $\mathbb{K}$ as
\begin{equation}
  \mathcal{W}_{\psi_N}[\eta](g) = \prod_{\nu} \mathcal{W}_{\psi_N,\nu}[\eta](g_\nu), \qquad g = \prod_\nu g_\nu, \qquad g_\nu \in \mathbf{G}(\mathbb{K}_\nu),
\end{equation}
where $\mathcal{W}_{\psi_N,\nu}[\eta]$ are local Whittaker functions. This follows from the fact that the corresponding local Whittaker models induced by $\psi_N$ are at most one-dimensional \cite{GK:1975, Shalika, Rod, Kostant}.

More generally, to any unipotent subgroup $\mathbf{U}\subset \mathbf{G}$ and any unitary character $\psi_U$ on $U = \mathbf{U}(\mathbb{A})$ trivial on $\mathbf{U}(\mathbb{K})$ one can associate the Fourier coefficient
(or `unipotent period integral')
\begin{equation}
\label{eq:Ucoeff}
\mathcal{F}_{\psi_U}[\eta](g) =\int_{\mathbf{U}(\mathbb{K})\backslash \mathbf{U}(\mathbb{A})} \eta(ug)\psi_U(u)^{-1}du.
\end{equation}
In particular, when $\mathbf{U}$ is the unipotent radical of a standard parabolic subgroup $\mathbf{P}$ of $\mathbf{G}$, we will call $\mathcal{F}_{\psi_U}$ a \emph{parabolic Fourier coefficient}.
The above Whittaker coefficients correspond to the special case when $\mathbf{P}$ equals the minimal parabolic $\mathbf{B}$.

It is a difficult question to determine, in general, for which choice of data $(\eta, \mathbf{U}, \psi_U)$ the Fourier coefficient $\mathcal{F}_{\psi_U}[\eta]$ is Eulerian.
Generally, one does \emph{not} expect it to factorize.
However, when $\eta$ belongs to a class of \emph{minimal automorphic representations} $\pi_{\text{min}}$, obtained as special values (see Table~\ref{tab:SmallReps}) of spherical Eisenstein series on groups of ADE-type, we recently showed that $\mathcal{F}_{\psi_U}[\eta]$ is Eulerian if $U$ is the unipotent radical of a maximal parabolic subgroup of $G$, and $\psi_U$ is a non-trivial character on $U$ \cite{LeviDist, NTMSimplyLaced}.\footnote{Previous related work includes \cite{GRS2,GanSavin,MillerSahi,GMV,GKP,FGKP}.}

In this paper we prove Eulerianity for a large class of Fourier coefficients. We will use a variety of techniques, building on our previous works \cite{LeviDist,NTMSimplyLaced} as well as on \cite{GGS} and \cite{FKP,FGKP}. To explain our results we first introduce some terminology which will be detailed further in \S\ref{sec:prel}.

Let $\mathfrak{g} = \mathfrak{g}_\mathbb{K}$ be the Lie algebra of $\Gamma = \mathbf{G}(\mathbb{K})$ and $\mathfrak{g}^*$ its vector space dual. A \emph{Whittaker pair} is an ordered pair $(S, \varphi) \in \mathfrak{g} \times \mathfrak{g}^*$ where $S$ is semi-simple and rational, i.e.\ $\ad(S)$ has eigenvalues in $\mathbb{Q}$,  and $\text{ad}^{*}(S)\varphi=-2\varphi$.
Fix an additive, unitary character $\chi$ on $\mathbb{A}$ trivial on $\mathbb{K}$. A Whittaker pair $(S, \varphi)$ determines a unipotent subgroup $N_{S, \varphi} \subset \mathbf{G}(\mathbb{A})$ defined in~\eqref{eq:N-S-phi} along with a unitary character $\chi_\varphi$ on $N_{S, \varphi}$, trivial on $N_{S, \varphi} \cap \Gamma$, defined by $\chi_\varphi(n)\coloneqq \chi(\varphi(\log n))$.
In this way we can associate to any automorphic form $\eta$ and Whittaker pair $(S, \varphi)$ a Fourier coefficient
\begin{equation}
  \mathcal{F}_{S,\varphi}[\eta](g)=\int_{(N_{S, \varphi} \cap \Gamma)\backslash N_{S, \varphi}} \eta(ng)\chi_\varphi(n)^{-1}dn.
\end{equation}
Note that this class of Fourier coefficients includes, in particular, all \emph{parabolic Fourier coefficients} as well as the coefficients studied in~\cite{GRS2,GRS,Ginz,GH,JLS} which we call \emph{neutral Fourier coefficients} and will define below.

The vanishing properties of these neutral Fourier coefficients determine the \emph{global wave-front set} $\WO(\eta)$ consisting of $\Gamma$-orbits of nilpotent elements $\varphi \in \mathfrak{g}^*$.
Using the partial ordering of these orbits, the global wave-front set is used to define minimal and next-to-minimal representations in \S\ref{sec:prel}.
The set of maximal orbits in $\WO(\eta)$ form the \emph{Whittaker support} $\WS(\eta)$ and we say that Fourier coefficients $\mathcal{F}_{S,\varphi}[\eta]$ with $\Gamma \varphi \in \WS(\eta)$ are of \emph{maximal rank}.

To present our results in full generality, we will need a more general notion of Fourier coefficients which, besides $N_{S,\varphi}$, contain a further period integral depending on a choice of an isotropic subspace of $\mathfrak{g}$.
Let us briefly describe it now, and refer to Definition~\ref{def:extended-Whittaker-Fourier-coefficient} below for further details.

For a Whittaker pair $(S,\varphi)$ let $\fu_S$ denote the nilpotent subalgebra of $\fg$ that equals the sum of all eigenspaces of $\ad(S)$ in $\fg$ corresponding to eigenvalues at least 1.
Then $\varphi$ defines an antisymmetric form $\omega_{\varphi}$ on $\fu_S$ by $\omega_{\varphi}(X,Y)\coloneqq \varphi([X,Y])$.
The Lie subalgebra $\fn_{S,\varphi}$ of $N_{S,\varphi} \cap \Gamma$ mentioned above is the radical of this form.
Choose any isotropic subspace $\mathfrak{i}$ of $\mathfrak{u}_S$ with respect to $\omega_\varphi$ that contains $\mathfrak{n}_{S,\varphi}$ and let $I\coloneqq \Exp(\mathfrak{i}\otimes_\mathbb{K} \mathbb{A})$ be the corresponding subgroup of $G = \mathbf{G}(\mathbb{A})$. For any automorphic form $\eta$, we define the \emph{isotropic Fourier coefficient} $\mathcal{F}^I_{S,\varphi}[\eta]$ by
\begin{equation}
  \mathcal{F}^I_{S,\varphi}[\eta](g)=\int_{(I \cap \Gamma)\backslash I} \eta(ng)\chi_\varphi(n)^{-1}dn,
\end{equation}
where $\chi_\varphi$ can be extended to a character on $I$ as explained in \S\ref{sec:prel}.
We are in particular interested in maximal isotropic subspaces of $\mathfrak{u}_S$ with respect to inclusion.
The corresponding maximal isotropic Fourier coefficients will be called \emph{Fourier--Jacobi coefficients} for short and are a generalization of the similarly named coefficients in~\cite[\S 5.2.3]{HundleySayag} (c.f.\ \cite{Ikeda,GRS:manuscripta}).
As further discussed in Remark~\ref{rem:maximal-isotropic}, maximal isotropic subspaces are in bijection with Lagrangian subspaces of the abelian quotient $\fu_S/\fn_{S,\varphi}$.

If $\ad(S)$ does not have eigenvalue 1 then $\fu_S=\fn_{S,\varphi}$ and thus $\mathcal{F}_{S,\varphi}[\eta]=\mathcal{F}^I_{S,\varphi}[\eta]$ meaning that $\mathcal{F}_{S,\varphi}$ is a Fourier--Jacobi coefficient.
This is, for example, the case for any parabolic Fourier coefficient.
In general, $\mathcal{F}_{S,\varphi}[\eta]$ equals a sum of left-translates of~$\mathcal{F}^I_{S,\varphi}[\eta]$ as seen in \cite[Lemma~3.1.1]{LeviDist}.

The main results of this paper are contained in the following statements.

\begin{maintheorem}[Transfer of Eulerianity, \S\ref{subsec:Trans}]
  \label{thm:transfer}
Let $\eta$ be an automorphic form on $G$.
Let $(S,\varphi)$ and $(H,\psi)$ be two Whittaker pairs such that $\Gamma \psi= \Gamma \varphi\in \WS(\eta)$.
Suppose that a Fourier--Jacobi coefficient $\cF^I_{S,\varphi}[\eta]$ is Eulerian. Then any Fourier--Jacobi coefficient $\cF^{I'}_{H,\psi}[\eta]$ is also Eulerian.
\end{maintheorem}

In particular, we may use this to transfer Eulerianity to Fourier coefficients $\mathcal{F}_{S,\varphi}$ for which $\ad(S)$ has no eigenvalue $1$, such as to parabolic Fourier coefficients.
Recall that any such Fourier coefficient is a Fourier--Jacobi coefficient.

\begin{maintheorem}[Hidden symmetry, \S\ref{sec:hidden}]
  \label{thm:Triv}
  Let $\eta$ be an automorphic form on $G$, and let $(H,\varphi)$ be a Whittaker pair with $\Gamma\varphi\in\WS(\eta)$.
  Then any unipotent element of the centralizer of the pair $(H,\varphi)$ in $G$ acts trivially on the Fourier coefficient $\cF_{H,\varphi}[\eta]$ using the left regular action.
\end{maintheorem}

From Theorem~\ref{thm:transfer} and the local uniqueness results in \cite{LokeSavin2006,KobayashiSavin}  we obtain the following corollary.

\begin{maincorollary}[\S\ref{sec:min-rep}]
  \label{cor:min-rep}
Let $\bfG$  be split of type $D_n$ or $E_7$,
and let $\pi$ be an irreducible  unitary minimal subrepresentation
of $\cA(G)$.
Then any Fourier--Jacobi coefficient $\cF_{H,\varphi}^I$ with $\varphi\neq 0$ is Eulerian on $\pi$.
\end{maincorollary}

From Theorems~\ref{thm:transfer} and~\ref{thm:Triv} together with local uniqueness results from~\cite{GGP,DSZ} we have the following corollary.

\begin{maincorollary}[\S\ref{sec:Bessel}]
  \label{cor:EuAD}
  Suppose that $\mathbf{G}$ is split of type $B_n$ or $D_n$ and let $\pi$ be an irreducible admissible automorphic representation of $\mathbf{G}(\mathbb{A})$.
  Let $(H,\varphi)$ be a Whittaker pair with $\Gamma\varphi\in\WS(\pi)$ such that the complex orbit of $\varphi$ is $(31\ldots1)$. Then any Fourier--Jacobi coefficient $\cF^I_{H,\varphi}$ is Eulerian on~$\pi$.
\end{maincorollary}

The next result concerns a class of spherical Eisenstein series and their Fourier--Jacobi coefficients (in particular parabolic Fourier coefficients) that are of interest in string theory as reviewed in~\cite{FGKP}.
The result follows from analyzing degenerate Whittaker coefficients using the reduction formula of Theorem~\ref{thm:reduction}, together with Theorem~\ref{thm:transfer}.

\begin{maintheorem}[\S\ref{sec:Eisenstein}]
  \label{thm:ntm-Eisenstein}
  Let $\mathbf{G}$ be a simple algebraic group split and defined over $\mathbb{Q}$ of type A, D, or E.  Then, for a class of spherical Eisenstein series in minimal and next-to-minimal automorphic representations of $\mathbf{G}(\mathbb{A}_\mathbb{Q})$, listed in Table~\ref{tab:SmallReps}, all Fourier--Jacobi coefficients of maximal rank are Eulerian.
\end{maintheorem}

For $G=\GL_n(\A)$, we deduce from \cite{MogWalds,Ginz,GGS} a very general theorem.

\begin{maintheorem}[\S\ref{subsec:GL}]\label{thm:GL}
Let $\pi$ be an irreducible admissible automorphic representation in the discrete spectrum
of $\GL_n(\A)$. Let $(H,\psi)$ be a Whittaker pair with $\Gamma\psi\in \WS(\pi)$. Then any Fourier--Jacobi coefficient $\cF_{H,\psi}^I$ is Eulerian on $\pi$.
\end{maintheorem}

\begin{remark}
  We expect that Theorems~\ref{thm:ntm-Eisenstein} and~\ref{thm:GL} hold in wider generality for other automorphic forms and other groups as well. We remark here that for a neutral pair $(H,\varphi)$, the dimension of $I$ is half the dimension of the orbit of $\varphi$, and thus this expectation is in compliance with the dimension formula of \cite{GinzDim}.
\end{remark}

\begin{remark}
  A local result related to Theorem~\ref{thm:transfer} is proven in \cite{MW}. Namely, \cite{MW} consider non-archimedean local counterparts of Fourier coefficients, called degenerate Whittaker models.
  They prove that for all Whittaker pairs $(H,\varphi)$ with $\varphi$ in the local Whittaker support, the dimension of the corresponding local Whittaker model depends only on the orbit of $\varphi$.

While local uniqueness implies Eulerianity, the result of \cite{MW} cannot in general replace the use of Theorem \ref{thm:transfer} for two reasons.
First, \cite{MW} do not consider the archimedean places, and second, the local Whittaker support (which equals the local wave-front set) can be bigger than the global one.

\end{remark}


After introducing the necessary preliminary background material in \S\ref{sec:prel}, we begin in \S\ref{sec:thm} by establishing some general results including transfer of Eulerianity (Theorem~\ref{thm:transfer}) and hidden invariance of Fourier coefficients (Theorem~\ref{thm:Triv}) with respect to the left regular action using techniques developed in \cite{LeviDist,GGS:support}.
In the same section, we also show Eulerianity of Fourier coefficients for $\GL_n(\mathbb{A})$ (Theorem~\ref{thm:GL}).
In \S\ref{sec:applications-to-small-reps} we apply Theorem~\ref{thm:transfer} to Fourier coefficients for small automorphic representations.
Specifically, in \S\ref{sec:min-rep} we use the local uniqueness results of~\cite{KobayashiSavin} to study Eulerianity in the case of minimal representations (Corollary~\ref{cor:min-rep}).
In \S\ref{sec:Bessel} we use Theorem~\ref{thm:Triv} and the local uniqueness results of so-called \emph{Bessel models} \cite{GGP,DSZ} to prove that, for next-to-minimal automorphic representations of $D_n$, Fourier--Jacobi coefficients with characters in the orbit with partition $(31^{2n-3})$ are Eulerian.

In \S\ref{sec:Eisenstein} we focus on certain spherical Eisenstein series realizing minimal and next-to-minimal automorphic representations.
We show that degenerate Whittaker coefficients of maximal rank for that representation are Eulerian by direct computation using the reduction formula of Theorem~\ref{thm:reduction} from~\cite{FKP}.
We then use Theorem~\ref{thm:transfer} to transfer Eulerianity to other Fourier--Jacobi coefficients, including other parabolic coefficients.

\medskip

\noindent\textbf{Acknowledgements.}
We are grateful for helpful discussions and correspondence with Guillaume Bossard, Solomon Friedberg, Gordan Savin and Eitan Sayag.
We especially thank Gordan Savin for his extensive explanations on minimal representations.

D.G.\ was partially supported by ERC StG grant 637912 and BSF grant 2019724. H.G.\ and D.P.\ were supported by the Swedish Research Council (Vetenskapsr\aa det), grants 2018-06774 and 2018-04760 respectively.
S.S.\ was partially supported by NSF grants DMS-1939600 and DMS-2001537, and Simons' foundation grant 509766.

\section{Preliminaries}
\label{sec:prel}

Let $\mathbb{K}$ be a number field, $\mathfrak{o}$ its ring of integers and $\mathbb{A}$ its ring of adeles. Let $\mathbf{G}$ be a reductive algebraic group defined over $\mathbb{K}$, $G = \mathbf{G}(\mathbb{A})$ and $\Gamma = \mathbf{G}(\mathbb{K})$.
Let $\mathfrak{g}_\mathbb{K}$ be the Lie algebra of $\mathbf{G}(\mathbb{K})$ which we will often abbreviate to $\mathfrak{g}$.
We will use a similar notation for the Lie algebra of any subgroup of $\mathbf{G}(\mathbb{K})$.

Conversely, let $\mathfrak{u} = \mathfrak{u}_\mathbb{K}$ be a nilpotent Lie subalgebra of $\mathfrak{g}$.
It defines a unipotent algebraic subgroup $\mathbf{U}$ of $\mathbf{G}$ such that $\mathbf{U}(\mathbb{K}) = \Exp(\mathfrak{u}_\mathbb{K})$ and $\mathbf{U}(\mathbb{A}) = \Exp(\mathfrak{u}_\mathbb{K} \otimes_\mathbb{K} \mathbb{A})$.
For convenience we will use the notation $U \coloneqq \mathbf{U}(\mathbb{A})$ and $[U] \coloneqq (U \cap \Gamma) \bs U = \mathbf{U}(\mathbb{K}) \bs \mathbf{U}(\mathbb{A})$.

Let $\mathbb{K}_\nu$ denote the completion of $\mathbb{K}$ with respect to a place $\nu$ and $\mathfrak{o}_\nu$ the corresponding completion of the ring of integers $\mathfrak{o}$.
For the following definition we introduce the notation $K_\text{fin} \coloneqq \prod_{\text{finite } \nu} \mathbf{G}(\mathfrak{o}_\nu)$ and $G_\text{inf} \coloneqq \prod_{\text{infinite } \nu} \mathbf{G}(\mathbb{K}_\nu)$.

\begin{definition}
  \label{def:automorphic-form}
  An \emph{automorphic form} on $G = \mathbf{G}(\mathbb{A})$ is a function on $G$ that is left $\Gamma\text{-invariant}$, finite under the right-action of $K_\text{fin}$, and smooth under the right-action of $G_\text{inf}$.
  Furthermore, it should be of moderate growth, and finite under the action of the center $\mathcal{Z}$ of the universal enveloping algebra of $\mathfrak{g}$  (see e.g. \cite{FGKP} for details).
  We will denote the space of automorphic forms on $G$ by $\mathcal{A}(G)$.
\end{definition}

Some of our intermediate results hold for a space of functions more general than $\mathcal{A}(G)$, where the requirements of moderate growth and finiteness under $\mathcal{Z}$ are dropped.
We call these functions \emph{automorphic functions}.
See \cite{LeviDist} for details.

We are interested in different types of Fourier coefficients of automorphic functions. A convenient way of describing different Fourier coefficients is via with so-called \emph{Whittaker pairs} which we will now introduce.
Let $S$ be a semi-simple element in $\mathfrak{g}$.
We say that $S$ is \emph{rational} if $\ad S$ has eigenvalues in $\mathbb{Q}$.
Let also $\mathfrak{g}^*$ be the vector space dual of $\mathfrak{g}$.
\begin{definition}
  \label{def:whittaker-pair}
  A \emph{Whittaker pair} is an ordered pair $(S, \varphi) \in \mathfrak{g} \times \mathfrak{g}^*$ where $S$ is rational semi-simple and $\ad^*(S)\varphi = -2\varphi$.
\end{definition}
The fact that $\varphi$ is in the $(-2)$-eigenspace of $\ad^*(S)$ implies that $\varphi$ is nilpotent
in the sense that the closure of its co-adjoint orbit contains the trivial element. We shall also use that it is paired with a nilpotent element $f_\varphi\in\mathfrak{g}$ using the Killing form. For reductive groups, the Killing form can be degenerate but there is a unique nilpotent element that it is paired with.

Since eigenspaces of rational semi-simple elements will figure frequently throughout the paper we introduce the following notation.
We denote by $\mathfrak{g}^S_\lambda$ the eigenspace of $\ad(S)$ in $\mathfrak{g}$ with eigenvalue $\lambda \in \mathbb{Q}$.
Let also $\mathfrak{g}^S_{>\lambda} \coloneqq \bigoplus_{\mu > \lambda} \mathfrak{g}^S_\mu$ and similarly for other inequalities as well as for $\mathfrak{g}$ and $\ad(S)$ replaced by $\mathfrak{g}^*$ and $\ad^*(S)$.
Moreover, for $\varphi \in \mathfrak{g}^*$ we also denote by $\mathfrak{g}_\varphi$ the stabilizer of $\varphi$ in $\mathfrak{g}$ under the coadjoint action.

For a Whittaker pair $(S, \varphi)$ let $\mathbf{N}_{S,\varphi}$ be the unipotent subgroup of $\mathbf{G}$ defined by its Lie algebra
\begin{equation}
  \label{eq:N-S-phi}
  \mathfrak{n}_{S,\varphi} = \mathfrak{g}^S_{>1} \oplus (\mathfrak{g}^S_1 \cap \mathfrak{g}_\varphi) \, .
\end{equation}
via $N_{S,\varphi} \coloneqq \mathbf{N}_{S,\varphi}(\mathbb{A}) = \Exp(\mathfrak{n}_{S,\varphi} \otimes_\mathbb{K} \mathbb{A})$.
Then, $\varphi$ defines a character $\chi_\varphi$ on $[N_{S,\varphi}] \coloneqq \mathbf{N}_{S,\varphi}(\mathbb{K}) \bs \mathbf{N}_{S,\varphi}(\mathbb{A})$ as follows.
Fix a non-trivial unitary character $\chi$ on $\mathbb{A}$ trivial on $\mathbb{K}$  and define $\chi_\varphi$ by $\chi_\varphi(n) = \chi(\varphi(\log n))$ for $n \in N_{S,\varphi}$ where the logarithm is well-defined since $n$ is unipotent.

\begin{definition}
  \label{def:fourier-coeff}
  We define the \emph{Fourier coefficient} of an automorphic form $\eta \in \mathcal{A}(G)$ with respect to a Whittaker pair $(S, \varphi)$ as
  \begin{equation}
    \mathcal{F}_{S,\varphi}[\eta](g) \coloneqq \intl_{[N_{S,\varphi}]} \eta(ng) \chi_\varphi(n)^{-1} \, dn
  \end{equation}
  where $g \in G_\mathbb{A}$ and $dn$ denotes the pushforward of the Haar measure of $N_{S,\varphi}$ normalized to be a probability measure.
  If $(S,\varphi) = (0,0)$ the corresponding Fourier coefficient is defined to be the automorphic form $\eta$ itself.
\end{definition}
Although the above integral is over the compact space $[N_{S,\varphi}]$, we will show that for certain $(S,\varphi)$ and $\eta$ it is \emph{Eulerian} meaning that it factorizes as an Euler product over the places $\nu$ of $\mathbb{K}$.

\begin{definition}
  We say that a function $f : \mathbf{G}(\mathbb{A}) \to \mathbb{C}$ is \emph{Eulerian} if there exist functions $f_\nu : \mathbf{G}(\mathbb{K}_\nu) \to \mathbb{C}$ for each place $\nu$ of $\mathbb{K}$ such that $f(g) = \prod_{\nu} f_\nu(g_\nu)$ for all $g = \prod_\nu g_\nu \in \mathbf{G}(\mathbb{A})$ where $g_\nu \in \mathbf{G}(\mathbb{K}_\nu)$.

  For an irreducible admissible automorphic representation $\pi \subset \mathcal{A}(\mathbf{G}(\mathbb{A}))$, we will say that a Fourier coefficient $\mathcal{F}_{S,\varphi}$ is Eulerian on $\pi$ if for every pure tensor $\eta \in \pi$,  $\mathcal{F}_{S,\varphi}[\eta] : \mathbf{G}(\mathbb{A}) \to \mathbb{C}$ is an Eulerian function.
\end{definition}

The Eulerianity can be checked on any pure tensor, in particular on the spherical vector, by the following lemma.

\begin{lemma}
If $\mathcal{F}_{S,\varphi}[\eta]$ is Eulerian for some pure tensor $\eta\in \pi$ then $\mathcal{F}_{S,\varphi}$ is Eulerian on $\pi$.
\end{lemma}
\begin{proof}
Let $\eta'\in \pi$ be another pure tensor. By the definition of restricted tensor product, $\eta$ and $\eta'$ differ only at finitely many places. The statement can be proven by induction on the number $n$ of these places. The base is $n=0$. For the induction step, let $\nu$ be a  place with $\eta_{\nu}\neq \eta'_{\nu}$. Since  $\pi_{\nu}^{\infty}$ is an irreducible smooth representation of $\bfG(\K_{\nu})$, there exists a smooth compactly supported function $h$ on  $\bfG(\K_{\nu})$ such that $\pi_{\nu}(h)\eta_{\nu}=\eta'_{\nu}$. Define $\zeta\in \pi$ by $\zeta_{\nu}=\eta'_{\nu}$ and $\zeta_{\mu}=\eta_{\mu}$ for any $\mu\neq \nu$. Let $\mathcal{F}_{S,\varphi}[\eta]=\prod_{\mu} f_{\mu}$ be the Euler product decomposition of $\mathcal{F}_{S,\varphi}[\eta]$. Let $f'_{\nu}$ be the convolution $f_{\nu}*h$. Then we have $\mathcal{F}_{S,\varphi}[\zeta]=f'_{\nu}\prod_{\mu\neq \nu} f_{\mu}$, and thus $\mathcal{F}_{S,\varphi}[\zeta]$ is Eulerian. By the induction hypothesis, it follows that $\mathcal{F}_{S,\varphi}[\eta']$ is Eulerian.
\end{proof}

By \cite[Lemma~3.2.8]{LeviDist}, if $\gamma \in \Gamma = \mathbf{G}(\mathbb{K})$ then $\mathcal{F}_{S,\varphi}[\eta](g) = \mathcal{F}_{\Ad(\gamma)S, \Ad^*(\gamma)\varphi}[\eta](\gamma g)$.
Therefore, when considering vanishing properties of Fourier coefficients, it is convenient to consider orbits of Whittaker pairs.
In particular, it will be useful to consider pairs which are related to Jacobson--Morozov $\sl_2$-triples, whose conjugacy classes are in bijection with nilpotent orbits in $\mathfrak{g}$.

\begin{definition}
  \label{def:neutral-pair}
  A Whittaker pair $(h, \varphi)$ is called \emph{neutral} if there exists an $\sl_2$-triple $(e, h, f)$ with standard commutation relations and such that $f = f_\varphi$ is the Killing form pairing with $\varphi$.
  The pair $(h, \varphi) = (0,0)$ will, for convenience, also be called neutral.
\end{definition}

The Jacobson--Morozov theorem implies that for any nilpotent $\varphi \in \mathfrak{g}^*$ there exists a neutral pair $(h, \varphi)$, and all the neutral pairs with the same $\varphi$ are conjugate (see \cite{Bou}).
By \cite[Theorem C]{GGS},
if $\eta$ is an automorphic function and $\mathcal{F}_{h,\varphi}[\eta] \equiv 0$ for some neutral pair $(h,\varphi)$, then $\mathcal{F}_{S,\varphi}[\eta] = 0$ for any Whittaker pair with the same $\varphi$.
We therefore only need to characterize the vanishing properties for Fourier coefficients of neutral pairs.

\begin{definition}
  For an automorphic form $\eta$, we define the \emph{global wave-front set} $\WO(\eta)$ to be the set of nilpotent $\Gamma$-orbits $\Oh$ in $\mathfrak{g}^*$ for which there exists a neutral pair $(h, \varphi)$ with $\varphi \in \Oh$ such that $\mathcal{F}_{h, \varphi}[\eta]$ is non-vanishing.
  We also define the \emph{Whittaker support} $\WS(\eta)$ to be the set of maximal orbits in $\WO(\eta)$.

For an automorphic representation $(\pi,V)$ in the space of automorphic forms $\mathcal{A}(G)$
we define the \emph{global wave-front set} $\WO(\pi)$ for $\pi$ be the union $\bigcup_{\eta \in V} \WO(\eta)$, and the \emph{Whittaker support} $\WS(\pi)$ to be the set of maximal orbits in $\WO(\pi)$.
\end{definition}

\begin{remark}
\begin{enumerate}[label=\textnormal{(\roman*)}]
\item We use the partial ordering of $\Gamma$-orbits defined in \cite[\S 2.4]{LeviDist}.
  If $\Oh < \Oh'$ in this partial ordering, then, for any place $\nu$ of $\mathbb{K}$, the closure (in the local topology) of $\Oh'$ in $\mathfrak{g}^*(\mathbb{K}_\nu)$ contains $\Oh$ as shown in Lemma~2.4.2 of the same paper.
\item It is easy to see that if $\pi\subset \cA(G)$ is generated by an automorphic form $\eta$ then $\WO(\pi)=\WO(\eta)$ and $\WS(\pi)=\WS(\eta)$.
\end{enumerate}
\end{remark}
The $\mathbf{G}(\mathbb{K})$-orbits used above are more refined than the standard complex orbits discussed for example in \cite{CM}.
In \cite{LeviDist} this allowed for stronger statements relating Fourier coefficients associated to different $\mathbf{G}(\mathbb{K})$-orbits.
It is, however, useful to relate them to complex orbits for which we have a coarser partial ordering given by the Zariski closure and a classification determined by Bala--Carter labels.
The Bala--Carter label of a complex orbit $\Oh$ is determined by the Cartan type of the unique conjugacy class of minimal Levi subalgebras that has non-trivial intersection with $\Oh$.
When the Cartan type does not uniquely identify the orbit, the labels are further decorated as in the case for $D_n$, $n>4$ with the orbits $(2A_1)'$ and $(2A_1)''$ corresponding to the partitions $(31^{2n-3})$ and $(2^41^{2n-8})$ respectively, both having Cartan type $2A_1 = A_1 \times A_1$.
Note that the $(2A_1)'$ orbit has representatives on the sum of the simple root spaces $\alpha_{n-1}$ and $\alpha_n$ and intersects the unipotent radical of the standard parabolic with semi-simple Levi $D_{n-1}$ while the $(2A_1)''$ orbit does not.
For $D_4$, there are two orbits corresponding to the partition $(2^4)$, because it is \emph{very even}, and the three next-to-minimal orbits are interchangeable by triality.
These complex orbits will be used to define minimal and next-to-minimal representations of $G$ in the space of automorphic forms~$\mathcal{A}(G)$.

We fix an embedding $\mathbb{K} \into \mathbb{C}$ with which we define $\mathfrak{g}_\mathbb{C} = \mathfrak{g}_\mathbb{K} \otimes_\mathbb{K} \mathbb{C}$ and the corresponding complex Lie group $\mathbf{G}(\mathbb{C})$.
We may thus map any $\mathbb{K}$-rational orbit $\mathbf{G}(\mathbb{K})x$ to a corresponding complex orbit $\mathbf{G}(\mathbb{C})x$ and by~\cite{Dok} this complex orbit does not depend on the choice of embedding.

\smallskip
\noindent
\begin{definition}
  \label{def:min-and-ntm-orbits}
  Let $\Oh$ be a non-zero $\mathbf{G}(\mathbb{C})$-orbit in $\mathfrak{g}_\mathbb{C}$ (or $\mathfrak{g}_\mathbb{C}^*$).
  We say that $\Oh$ is \emph{minimal} if its Zariski closure is a disjoint union of $\Oh$ itself and the zero orbit.
\end{definition}

\begin{definition}\label{def:ntm-orbits}
  For a semi-simple, simply-laced group $\mathbf{G}$ we will say that an orbit $\Oh\subset \fg_\C$ is \emph{next-to-minimal} if its Zariski closure is a disjoint union of $\Oh$ itself, minimal orbits, and the zero orbit with the additional condition that $\Oh$ does not intersect any component of $\mathfrak{g}_\mathbb{C}$ of type $A_2$.
\end{definition}

The same terminology is applied to $\mathbf{G}(\mathbb{K})$-orbits via the map to complex orbits described above.

If $\mathfrak{g}$ is simple and simply-laced then \cite[Lemma 2.1.1]{NTMSimplyLaced} gives that $\Oh$ is minimal if and only if it has Bala--Carter label $A_1$ and next-to-minimal if and only if it has Bala--Carter label $2A_1$ (possibly with decorations for type D as noted above).
The extra condition in Definition~\ref{def:ntm-orbits} for a next-to-minimal orbit is added to exclude for example the regular orbit for any $A_2$-type component of $\mathfrak{g}$.
Such an orbit has Bala--Carter label $A_2$ and the associated Fourier coefficients behave very differently.

Consider $\mathbf{G}(\mathbb{K})$-orbits in $\mathfrak{g}^*$ and denote by $\op{O}_{\{0\}}$ the set containing only the zero orbit, by $\op{O}_{\{1\}}$ the union of $\op{O}_{\{0\}}$ and the set of minimal orbits, and by $\op{O}_{\{2\}}$ the union of $\op{O}_{\{1\}}$ and the set of next-to-minimal orbits.

\begin{definition}
     \label{def:small}
  An automorphic form $\eta \in \mathcal{A}(G)$ is called \emph{minimal} if $\WO(\eta)$ is a subset of $\op{O}_{\{1\}}$ but not of $\op{O}_{\{0\}}$.
  It is called \emph{next-to-minimal} if $\WO(\eta)$ is a subset of $\op{O}_{\{2\}}$ but not of $\op{O}_{\{1\}}$.
  The same terminology can be applied to an automorphic representation $\pi$ by replacing $\WO(\eta)$ above with $\WO(\pi)$.
\end{definition}

\begin{definition}\label{def:dominate}
    Let $(H,\varphi)$ and $(S,\varphi)$ be two Whittaker pairs with the same $\varphi$. We say that $(H,\varphi)$ \emph{dominates} $(S,\varphi)$ if $H$ and $S$ commute and
\begin{equation}\label{=domin}
        \fg_{\varphi}\cap \fg^H_{\geq 1}\subseteq \fg^{S-H}_{\geq0}\,.
    \end{equation}
\end{definition}

\begin{lemma}[{\cite[Lemma~2.3.7 and Lemma~3.2.1]{LeviDist}}]
  \label{lem:neutral-dominates}
Let $(h,\varphi)$ be a neutral Whittaker pair, and let $Z\in \fg$ be a rational semi-simple element that commutes with $h$ and with $\varphi$. Then $(h,\varphi)$ dominates $(h+Z,\varphi)$.
Furthermore, for any Whittaker pair $(H, \varphi)$ there exists a neutral pair $(h, \varphi)$ and $Z \in \mathfrak{g}$ as above such that $H = h + Z$.
\end{lemma}

We will also use a more general class of Fourier coefficients.
Let $(S,\varphi)$ be a Whittaker pair, $\fu_S \coloneqq  \mathfrak{g}^S_{\geq1}$ and let $\fn_{S,\varphi}$ be as in~\eqref{eq:N-S-phi}.
Define an anti-symmetric form $\omega_{\varphi}$ on $\fu_S$ by $\omega_{\varphi}(X,Y)\coloneqq \varphi([X,Y])$.
Let $\mathfrak{i} \subseteq \fu_S$ be any isotropic subspace with respect to $\omega_\varphi$ that includes~$\lie n_{S,\varphi}$. Note $\fn_{S,\varphi}$ and $\mathfrak{i}$ are ideals in $\fu_S$.
In fact, $\mathfrak{n}_{S,\varphi}$ is the radical of the restriction $\omega_\varphi|_{\mathfrak{u}_S}$ as can be seen by \cite[Lemma~3.2.6]{GGS}.
Let $I$ be the subgroup $\Exp(\mathfrak{i} \otimes_\mathbb{K} \mathbb{A})$ of $G$ and $\mathbb{T} = \{z \in \mathbb{C} : \abs{z} = 1\}$.
Then $\chi_{\varphi}^I : I \to \mathbb{T}$ defined by $\chi_\varphi^I(x)=\chi(\varphi(\log x))$ is a character of $I$ trivial on~$I \cap \Gamma$ that extends $\chi_\varphi$ on $N_{S,\varphi}\subset I$.
Indeed, $\varphi(X) \in \mathbb{K}$ for $X \in \mathfrak{i} \subset \mathfrak{g}_\mathbb{K}$ and since $\mathfrak{i}$ is isotropic, we have that $\omega_\varphi|_{\mathfrak{i} \otimes_\mathbb{K} \mathbb{A}} = 0$ and thus $\chi_\varphi^I \in \Hom(I, \mathbb{T})$.
\begin{definition}
  \label{def:extended-Whittaker-Fourier-coefficient}
  For any $\eta\in \mathcal{A}(G)$, Whittaker pair $(S,\varphi)$ and $I = \Exp(\mathfrak{i} \otimes_\mathbb{K} \mathbb{A})$ as above, we define the corresponding \emph{isotropic Fourier coefficient} as
  \begin{equation}
    \label{eq:extended_Whittaker-Fourier_coefficient}
    \cF_{S,\varphi}^{I}[\eta](g)\coloneqq \intl_{[I]}\eta(ng)\, \chi_{\varphi}^{I}(n)^{-1}\, dn.
  \end{equation}
  If $\mathfrak{i}$ is a maximal isotropic subspace of $\mathfrak{u}_S$ with respect to inclusion we call $\mathcal{F}^I_{S,\varphi}$ a \emph{Fourier--Jacobi coefficient}.
  In the case when $(S,\varphi)$ is neutral these maximal isotropic coefficients are exactly those defined in~\cite[\S 5.2.3]{HundleySayag}, but our notion is more general.
\end{definition}

In general, $\cF_{S,\varphi}^{I}[\eta]$ is a further  period integral of $\cF_{S,\varphi}[\eta]$.

\begin{remark}
  \label{rem:maximal-isotropic}
  Note that the radical $\mathfrak{n}_{S,\varphi}$ is always an isotropic subspace of $\mathfrak{u}_S \coloneqq  \mathfrak{g}^S_{\geq 1}$, but under certain conditions it is also maximal.
  Note that $\omega_\varphi$ defines a symplectic form on the abelian quotient $\mathfrak{u}_S / \mathfrak{n}_{S,\varphi}$.
  Maximal isotropic subspaces $\mathfrak{i} \subseteq \mathfrak{u}_S$ with respect to $\omega_\varphi$ are in bijection with Lagrangian subspaces in $\mathfrak{u}_S/\mathfrak{n}_{S,\varphi}$ by taking their preimages in $\mathfrak{u}_S$.
  By \cite[Lemma~3.2.6]{GGS}, $\mathfrak{u} / \mathfrak{n}_{S,\varphi}$ is isomorphic to $\mathfrak{g}^S_1 / (\mathfrak{g}^S_1 \cap \mathfrak{g}_\varphi)$.
  If $\mathfrak{g}^S_1 = \{0\}$ then $\mathfrak{u}_S = \mathfrak{n}_{S,\varphi}$ and there is only one maximal isotropic subspace $\mathfrak{n}_{S,\varphi}$.
  Furthermore, if $(S,\varphi)$ is neutral then, from \cite[Lemma~3.4.3]{CM}, $\mathfrak{g}^S_1 \cap \mathfrak{g}_\varphi = \{0\}$ which means that $\mathfrak{u}_S / \mathfrak{n}_{S,\varphi} \iso \mathfrak{g}^S_1$.
\end{remark}

\begin{theorem}[{\cite[Theorem B, and Corollaries 3.1.2 and 4.3.2, and Proposition 4.3.3]{LeviDist}}]\label{thm:IntTrans}
  Let $\pi\subset \mathcal{A}(G)$ be any automorphic representation.
Let $(H,\varphi)$ and $(S,\varphi)$ be Whittaker pairs such that $(H,\varphi)$ dominates $(S,\varphi)$ and $\Gamma \varphi \in \mathrm{WS}(\pi)$.
  Then

    \begin{enumerate}[label=\textnormal{(\roman*)}]
        \item \label{itm:LinDet}
             $\cF_{H,\varphi}$ and $\cF_{S,\varphi}$ have the same kernel as linear operators from $\pi$ to $C^{\infty}(G)$.

        \item \label{itm:Int}
Let $\mathfrak{i} \subset \fg^H_{\geq 1}$ and $\mathfrak{i}'\subset \fg^S_{\geq 1}$ be maximal isotropic subspaces. Let
    \begin{equation}
        \mathfrak{v}\coloneqq  \mathfrak{i}/(\mathfrak{i}\cap \mathfrak{i}'),  \text{ and }\, \mathfrak{v}' \coloneqq \mathfrak{i}'/((\mathfrak{i}+\fg_{{\varphi}})\cap \mathfrak{i}')
    \end{equation}
    Let $I\coloneqq \Exp(\mathfrak{i} \otimes_\mathbb{K} \mathbb{A})$, $I'\coloneqq \Exp(\mathfrak{i}' \otimes_\mathbb{K} \mathbb{A})$, ${V}\coloneqq \Exp(\fv \otimes_\mathbb{K} \mathbb{A})$ and ${V'}\coloneqq \Exp(\fv' \otimes_\mathbb{K} \mathbb{A})$.
Then for any $\eta\in \pi$ we have
 \begin{equation}\label{=HfromS_easy}
                \cF^I_{H,\varphi}[\eta](g) =
                \intl{V} \cF^{I'}_{S,\varphi}[\eta](vg) \, dv \, \quad \text{and} \quad  \cF^{I'}_{S,\varphi}[\eta](g)=\intl{V'} \cF^I_{H,\varphi}[\eta](vg) \, dv.
            \end{equation}

    \end{enumerate}
\end{theorem}

\begin{proof}
  Statement~\ref{itm:LinDet} follows from~\cite[Theorem~B, and Proposition~4.3.3]{LeviDist} while~\ref{itm:Int} follows from \cite[Corollary 4.3.2]{LeviDist}, which proves the statement for a particular choice of maximal isotropic subspaces, together with \cite[Corollary~3.1.2]{LeviDist}, which relates different choices of maximal isotropic subspaces.

To be more precise, the isotropic subspaces considered in \cite[Corollary 4.3.2]{LeviDist} are not maximal, but they become maximal after one adds to them a maximal isotropic subspace of $\fg^S_1\cap \fg^H_1$, and the proof of \cite[Corollary 4.3.2]{LeviDist} work verbatim in this case.
\end{proof}

The statement and proof of Theorem~\ref{thm:IntTrans} also apply to a more general setup, considered in \cite{LeviDist}. In this setup  $G$ is a finite central extension of $\mathbf{G}(\mathbb{A})$, and $\pi$ lies in a more general space of functions on $G$, that we call automorphic functions.

\section{General results}
In this section we prove our general result on transfer of Eulerianity (Theorem~\ref{thm:transfer}). This will then be used in subsequent sections to deduce Eulerianity of (classes of) Fourier coefficients associated with minimal and next-to-minimal representations. In this section we also prove a result (Theorem~\ref{thm:Triv}) on the left-regular action of unipotent elements in $G$ on Fourier coefficients associated with Whittaker pairs. Both of these results also hold for automorphic functions and finite central extensions. At the end of the section we apply Theorem ~\ref{thm:transfer} to $\GL_n$, proving Theorem~\ref{thm:GL}.

\label{sec:thm}

\subsection{Transfer of Eulerianity}\label{subsec:Trans}

Let us recall the setup of Theorem~\ref{thm:transfer}.
We have an automorphic form $\eta$ on $G$ and two Whittaker pairs $(S, \varphi)$ and $(H, \psi)$ such that $\psi \in \Gamma \varphi \in \operatorname{WS}(\eta)$. We also have isotropic subspaces $\mathfrak{i} \subseteq \mathfrak{u}_S = \mathfrak{g}^S_{\geq1}$ and $\mathfrak{i}' \subseteq \mathfrak{u}_H = \mathfrak{g}^H_{\geq1}$ with respect to $\omega_\varphi$ and $\omega_\psi$ containing $\mathfrak{n}_{S,\varphi}$ and $\mathfrak{n}_{H,\psi}$ respectively.
We denote the corresponding unipotent subgroups in $G$ by $I$ and $I'$.
Given that $\mathcal{F}^I_{S,\varphi}[\eta]$ is Eulerian we will now prove that $\mathcal{F}^{I'}_{H,\psi}[\eta]$ is also Eulerian.
\begin{proof}[Proof of Theorem~\ref{thm:transfer}]
Let $(s,\varphi)$ and $(h,\psi)$ be neutral Whittaker pairs that dominate $(S,\varphi)$ and $(H,\psi)$ respectively.
Such pairs exist by Lemma~\ref{lem:neutral-dominates}.
By the theory of $\mathfrak{sl}_2$-triples (\cite[\S 11]{Bou}), there exists $\gamma\in \Gamma$ such that $(\Ad(\gamma)h,\Ad^*(\gamma)\psi)=(s,\varphi)$.
Let $\fr\subset \fu_s=\fg^s_{\geq 1}$ be any  maximal isotropic subspace with respect to $\omega_\varphi$.
Then $\fr'\coloneqq  \Ad(\gamma^{-1})\fr$ is a maximal isotropic subspace of~$\fu_h$ with respect to $\omega_\psi$. Let $R\coloneqq \Exp(\mathfrak{r} \otimes_\mathbb{K} \mathbb{A})$ and $R'\coloneqq \Exp(\mathfrak{r}' \otimes_\mathbb{K} \mathbb{A})$ be the corresponding unipotent subgroups of $G$.
Then, by a change of integration variable (cf.\ \cite[Lemma~3.2.8]{LeviDist}), we get that $\cF_{h,\psi}^{R'}[\eta]( g)=\cF_{s,\varphi}^{R}[\eta](\gamma g)$ for any $g \in G$.

Using Theorem~\ref{thm:IntTrans}~\ref{itm:Int}, we may now relate the above Fourier coefficients by (non-compact) adelic integral transforms:
\begin{equation}
  \mathcal{F}_{H,\psi}^{I'}[\eta]( g) \longleftrightarrow
  \mathcal{F}_{h,\psi}^{R'}[\eta](g) = \mathcal{F}_{s,\varphi}^{R}[\eta](\gamma g)  \longleftrightarrow \mathcal{F}^I_{S,\varphi}[\eta](\gamma g) \,.
\end{equation}
Since such transforms preserve Eulerianity, the theorem follows.
\end{proof}

\subsection{Hidden invariance}
\label{sec:hidden}
A Fourier coefficient $\mathcal{F}_{H,\varphi}[\eta](g)$ is $\chi_{\varphi}$-semi-invariant under left-translations of its argument $g$ by an element in the adelic unipotent subgroup $N_{H,\varphi} \subset G_\mathbb{A}$ as can be seen by a change of integration variable.
Similarly, if $\gamma \in \Gamma$, then \cite[Lemma~3.2.8]{LeviDist} shows that $\mathcal{F}_{H,\varphi}[\eta](g) = \mathcal{F}_{H',\varphi'}[\eta](\gamma g)$ where $H' = \Ad(\gamma)H$ and $\varphi' = \Ad^*(\gamma)\varphi$.
Thus, $\mathcal{F}_{H, \varphi}[\eta](g)$ is left unchanged if $\gamma$ centralizes~$(H, \varphi)$.
We will now show that there is a further, hidden invariance when $\Gamma\varphi \in \operatorname{WS}(\eta)$.

\begin{proof}[Proof of Theorem~\ref{thm:Triv}]
By \cite[Lemma 3.0.2]{GGS}, there exists an $\sl_2$-triple $\gamma=(e,h,f)$ such that $h$ commutes  with $H$, and $\varphi$ is given by the Killing form pairing with $f$.
Denote $\fq\coloneqq \fg^H_0\cap \fg_{\varphi}$ and $Z\coloneqq H-h$. Then $Z$ commutes with $\gamma$, and $\fq\subset \fg^h_{\leq 0}\cap \fg^H_0\subset \fg^Z_{\geq 0}$ since $\mathfrak{g}_\varphi \subset \mathfrak{g}^h_{\leq 0}$.
Thus, we can decompose $\mathfrak{q}$ into $Z$-eigenspaces as $\fq=\fq^Z_0\oplus \fq^Z_{>0}$ using a similar notation as for $\mathfrak{g}$.

Denote $\xi\coloneqq \cF_{H,\varphi}[\eta]$ which is a function on $G$.
Let us first show that for any $\lambda\in \Q_{>0}$ and any $u\in \Exp(\fq^Z_{\lambda} \otimes_\mathbb{K} \mathbb{A})$ that $\xi^u=\xi$, where $\xi^u(x)\coloneqq \xi(u^{-1}x)$ for $x\in G$.
Since $\varphi\in (\fg^*)^H_{-2}$ and $\fq \subset \fg^H_0$, $\varphi$ vanishes on $\fq$ and thus $\chi_{\varphi}(u)=1$.
Since $u$ centralizes the pair $(H,\varphi)$, we have that $u$ normalizes $N_{H,\varphi}$ and $\xi^u=\cF_{H,\varphi}[\eta^u]$ by a change of variable in the integral.
The measure remains unchanged since $u$ is unipotent.

Consider the Whittaker pair $(H+\lambda^{-1}Z,\varphi)$
and note that $u\in N_{H + \lambda^{-1}Z,\varphi}$ by construction. Thus, by shifts in the integration variable,
\begin{equation}\label{eq:uFix}
\cF_{H+\lambda^{-1}Z,\varphi}[\eta]=\cF_{H+\lambda^{-1}Z,\varphi}[\eta]^u=\cF_{H+\lambda^{-1}Z,\varphi}[\eta^u],
\end{equation}
meaning that $\eta^u - \eta$ is in the kernel of $\mathcal{F}_{H+\lambda^{-1} Z, \varphi}$.
The Whittaker pair $(H+\lambda^{-1}Z,\varphi)$ dominates the Whittaker pair $(H,\varphi)$. Since $\varphi\in \WS(\eta)$, Theorem~\ref{thm:IntTrans}~\ref{itm:LinDet} implies that the restrictions of
$\cF_{H+\lambda^{-1}Z,\varphi}$ and $\cF_{H,\varphi}$ to the representation generated by $\eta$ have the same kernel. From this and \eqref{eq:uFix} we obtain that $\xi^u-\xi=\cF_{H,\varphi}[\eta^u-\eta]=0$.

It remains to show that $\Exp(\mathfrak{q}^Z_0 \otimes_\mathbb{K} \mathbb{A})$ acts trivially on $\xi$.
By factorization, it is enough to show that $u = \exp(a e')$ where $a \in \mathbb{A}$ and $e' \in \mathfrak{q}^Z_0$ acts trivially.
Let $\mathfrak{g}^\gamma$ be the centralizer of the triple $\gamma$ in $\mathfrak{g}$.
The Lie algebra $\fq^Z_0$ is reductive, since it equals the centralizer of the semisimple element $Z$ in the reductive Lie algebra $\mathfrak{g}^\gamma$. By the Jacobson--Morozov theorem, $e'$ can be completed to an $\sl_2$-triple $(e',h',f')$ in $\mathfrak{q}^Z_0$.
By construction $u \in N_{h+h',\varphi}$ and, similar to above, we get that $\eta^u - \eta$ is in the kernel of $\cF_{h+h',\varphi}$.
By Lemma~\ref{lem:neutral-dominates}, the neutral Whittaker pair $(h,\varphi)$ dominates both $(h+h',\varphi)$ and~$(H,\varphi)$.
Thus, with repeated use of Theorem~\ref{thm:IntTrans}~\ref{itm:LinDet}, we get that $\eta^u - \eta$ is in the kernel of~$\cF_{H,\varphi}$ and hence $\xi^u - \xi = 0$.
\end{proof}

An application of this hidden invariance property is given by Corollary~\ref{cor:EuAD} that will be further discussed in \S\ref{sec:Bessel}.

\subsection{Eulerianity for type \texorpdfstring{$A_n$}{An}}\label{subsec:GL}

By the work of M\oe{}glin and Waldspurger \cite{MogWalds}, the discrete spectrum
of $\GL_n(\A)$ consists of Speh representations $\Delta(\tau,b)$ \cite{Speh}, where
$\tau$ runs over irreducible unitary cuspidal automorphic representations of $\GL_a(\A)$,
and $n = ab$. 
Let us show that any top Fourier coefficient of any Speh representation is Eulerian. By \cite[Proposition 5.3]{Ginz} or \cite{JL_GL}, the unique top orbit for  $\Delta(\tau,b)$ corresponds to the partition $a^b$. That is, $\WS(\Delta(\tau,b))=\{\cO_{a^b}\}$. By Theorem \ref{thm:transfer}, it is enough to prove Eulerianity for the Whittaker coefficient corresponding to this orbit, {\it i.e.} the coefficient $\cF_{H,\varphi}$, where $H$ is the diagonal matrix $H=\diag(n-1,n-3,\dots, 3-n,1-n)$, and $\varphi$ is given by the trace form pairing with the matrix in lower-triangular Jordan form, with $b$ blocks of size $a$ each.

\begin{proposition}\label{prop:GL}
$\cF_{H,\varphi}$ is Eulerian on $\Delta(\tau,b)$.
\end{proposition}
\begin{proof}
We prove by induction on $b$. For $b=1$, $\Delta(\tau,b)$ is a cuspidal representation, $\varphi$ is a regular nilpotent element, and $\cF_{H,\varphi}$ is factorizable by the local uniqueness of (non-degenerate) Whittaker models \cite{GK:1975,Shalika}.
For $b>1$, let $Q=LU$ denote the standard (upper-triangular) maximal parabolic with blocks $(b-1)a,a$.  Then $\cF_{H,\varphi}=\cF^L_{H,\varphi}\circ c_U$, where $c_U$ denotes the constant term, and $\cF^L_{H,\varphi}$ denotes the Whittaker coefficient corresponding to $(H,\varphi)$ on the Levi subgroup $L=\GL_{(a-1)b}(\A)\times \GL_{b}(\A)$. By \cite[Lemma 2.4]{Offen-Sayag}, $c_U(\Delta(\tau,b))$ is isomorphic as an automorphic representation of $L$ to the exterior product $\Delta(\tau,b-1)\boxtimes\tau$. It is easy to see that $\cF^L_{H,\varphi}$ also decomposes as an exterior product $\cF_{H_1,\varphi_1}\boxtimes \cF_{H_2,\varphi_2}$ of analogous Fourier coefficients. Explicitly,
\begin{equation}
H_1=\diag(a(b-1)-1,\dots, 1-a(b-1)), \quad H_2=\diag(a-1,\dots, 1-a),
\end{equation}
$\varphi_1$ is given by the trace form pairing with the matrix in lower-triangular Jordan form with $b-1$ blocks of size $a$ each, and $\varphi$ is given by a single Jordan block of size $a$.
Since $\cF_{H_1,\varphi_1}$ on $\Delta(\tau,b-1)$ and $\cF_{H_1,\varphi_1}$ on $\tau$ are Eulerian by the induction hypothesis, we obtain that 
$\cF_{H,\varphi}$ is Eulerian on $\Delta(\tau,b)$.
\end{proof}

Theorem~\ref{thm:GL} follows now from \cite{{MogWalds}}, Proposition \ref{prop:GL} and Theorem \ref{thm:transfer}.

\section{Applications to small representations}
\label{sec:applications-to-small-reps}

\subsection{Minimal representations}
\label{sec:min-rep}

We will use the following notions of minimal irreducible smooth representations of algebraic reductive groups over local fields $F$ of characteristic zero, following \cite{KobayashiSavin}.
An irreducible admissible smooth representation of a real reductive group is minimal if
its annihilator ideal in the universal enveloping algebra $\mathcal{U}(\fg_{\C})$ is the Joseph ideal. An irreducible representation of a  $p$-adic reductive group is  minimal if its
character, viewed as a distribution around 0 on $\fg(F)$, is equal to
$\widehat{\mu_{O_{\min}}}
 + c \delta_0$,
where $\widehat{\mu_{O_{\min}}}$ is the Fourier transform of the invariant measure on a minimal $G$-orbit in $\fg^*(F)$, $c\in \C$ and $\delta_0$ is the $\delta$-function at 0.

\begin{proof}[Proof of Corollary \ref{cor:min-rep}]
We can assume that $\varphi$ lies in the minimal orbit $\cO_{\text{min}} \subset \fg^*$, since otherwise $\cF^I_{H,\varphi}$ vanishes on $\pi$.
The group
$\mathbf{G}$ has a maximal parabolic subgroup $\mathbf{P}_\text{ab} = \mathbf{L}_\text{ab} \mathbf{U}_\text{ab}$ with abelian unipotent radical $\mathbf{U}_\text{ab}$.
Let $S_\text{ab} \in \fg$ be the Cartan element satisfying $\alpha(S_\text{ab}) = 0$ for the simple roots $\alpha$ that define $\mathbf{L}_\text{ab}$ and $\alpha(S_\text{ab}) = 2$ for the remaining simple roots.
Then $S_\text{ab}$ defines $\mathbf{U}_\text{ab}$ by \eqref{eq:N-S-phi} independent of character and  $\cO_{\text{min}}$ intersects $(\fg^*)^{S_\text{ab}}_{-2}$.
By \cite{MillerSahi}, this intersection lies in a single orbit under $\mathbf{P}_\text{ab}(\C)$.
Let $\psi$ lie in this intersection. Let us show that $\cF_{S_\text{ab},\psi}$ is Eulerian on $\pi$.

By Theorem \ref{thm:IntTrans}, $\cF_{S_\text{ab},\psi}$ is not identically zero on $\pi$, while $\cF_{S_\text{ab},\psi'}$ is identically zero on $\pi$ for any $\psi'\in(\fg^*)^{S_\text{ab}}_{-2}$ which does not lie in the minimal orbit.   Thus, $\pi$ is of global rank~$1$ in the sense of \cite[\S 9]{KobayashiSavin}. By \cite[Theorem 9.4]{KobayashiSavin}, this implies that for every place $\nu$, the local component $\pi_{\nu}$ of $\pi$ is of local rank 1. Since $\pi_{\nu}$ is unitary and has rank 1, it also has rank 1 with respect to the Heisenberg parabolic subalgebra by \cite[\S 2]{LokeSavin2006}. Now, \cite[Theorem 8]{LokeSavin2006} implies that $\pi_{\nu}$ is minimal.

Let $\chi$ be the character of $\mathbf{U}_\text{ab}(\mathbb{A})$ given by $\psi$.
  By \cite[Propositions 7.2 and 8.3]{KobayashiSavin} and \cite[Theorems A and B]{HKM}, for every $\nu$ with minimal $\pi_{\nu}$ we have
  $$\dim \bigl((\pi_{\nu}^{\infty})^*\bigr)^{\mathbf{U}_\text{ab}(\mathbb{K}_{\nu}),\chi}=1$$
  This implies that $\cF_{S_\text{ab},\psi}$ is Eulerian on $\pi$.
By Theorem \ref{thm:transfer}, so is $\cF^I_{H,\varphi}$.
\end{proof}

\begin{remark}
For the theorem we rely on the results of \cite{KobayashiSavin} which restrict the choice of group $\mathbf{G}$. We expect the theorem to be true more generally for other groups of type ADE. For example, for $E_8$ where one only has a Heisenberg parabolic the results of~\cite{Pollack} on (limits of) quaternionic discrete series could be helpful.
For $E_6$ the reference~\cite{MoellersSchwarz} might also be helpful.

We note that for types $B_n$ and $C_n$ there are no minimal automorphic representations, since the minimal orbits are not special in these cases, while Whittaker supports for automorphic representations of split classical groups contain only special orbits by \cite{JLS}.
Furthermore, for split real reductive groups of type $B_n$, $n\geq 4$, the absence of minimal irreducible admissible representations, even for the universal cover, is proven in~\cite{VoganSingular}.
\end{remark}

\begin{remark}\label{rmk:KS}
In \cite{KobayashiSavin} uniqueness of minimal representations was proved for a certain extended class of $\mathbb{K}$-forms of groups. The allowed $\mathbb{K}$-forms can be characterized by having a parabolic subgroup with abelian unipotent defined over $\mathbb{K}$. Over $\mathbb{K}=\mathbb{Q}$, this includes all split forms but only the quasi-split forms ${\rm SU}_{n,n}$ and ${\rm O}_{n,n+2}$. In the former case, the corresponding parabolic has abelian nilpotent consisting of $(n\times n$)-matrices that meets the minimal and next-to-minimal nilpotent orbit of ${\rm SU}_{n,n}$. In the latter case, the parabolic is associated with node 1 of the Dynkin diagram and the abelian nilpotent is the $(2n)$-dimensional representation of ${\rm O}_{n-1,n+1}$ that meets the minimal orbit and the next-to-minimal orbit $(2A_1)'$.
\end{remark}

\subsection{Next-to-minimal representations}
\label{sec:Bessel}

In this section we will use the results of \S\ref{sec:thm} to deduce Eulerianity of a large class of maximal rank Fourier coefficients of automorphic forms on groups of type B and D whose Whittaker support includes an orbit that embeds into the complex orbit of partition $(31\ldots1)$.
For type D this corresponds to a next-to-minimal orbit with Bala--Carter label $(2A_1)'$, while for type B the orbit has Bala--Carter label $B_1$ with a representative in the root space for the short simple root.

We will now deduce Corollary~\ref{cor:EuAD} from Theorems~\ref{thm:transfer} and~\ref{thm:Triv}, as well as from a local statement that follows from uniqueness of Bessel models.

Let $f\in \fg$ be such that $\varphi$ is given by the Killing form pairing with $f$. By \cite[\S5.2 and \S5.6]{MillerSahi} (cf. \cite[Lemma A.1]{LeviDist}) there exists a maximal torus and a root system such that $f$ lies in the root space of the short root $\sum_{i=1}^n \alpha_i$ for $B_n$, or in a direct sum of the root spaces of the roots $\sum_{i=1}^{n-1}\alp_i$ and $\sum_{i=1}^{n-2}\alp_i+\alp_n$, for $D_n$, in the Bourbaki notation. Let $S_{\alpha_1}$ be the rational semisimple element in $\mathfrak{g}$ given by the values $\alpha_i(S_{\alpha_1}) = 2$ if $i = 1$ and otherwise $0$.
Then $(S_{\alpha_1},\varphi)$ is a Whittaker pair.

Then $S_{\alpha_1}$ defines a unipotent subgroup $\mathbf{U}$ which is the unipotent radical of the maximal parabolic subgroup $\mathbf{P} \subset \mathbf{G}$ associated to the root $\alpha_1$.
Then the centralizer of $S_{\alpha_1}$ is a Levi subgroup $\mathbf{L}\subset \mathbf{P}$, and $\varphi$ is given by an element of the first internal Chevalley module of $\mathbf{L}$. Let $\mathbf{M} \subset \mathbf{L}$ denote the centralizer of~$\varphi$ in~$[\mathbf{L},\mathbf{L}]$. Then $\mathbf{M}$ is split and simple, and thus generated by its unipotent elements.
We have $[\mathbf{L},\mathbf{L}]\cong \SO_{n-1,n}$, and $\mathbf{M}\cong\SO_{n-1,n-1}$ for type B
and $[\mathbf{L},\mathbf{L}]\cong \SO_{n-1,n-1}$, and $\mathbf{M}\cong \SO_{n-1,n-2}$ for type D.
Let $\K_{\nu}$ be a completion of $\K$, archimedean or not.
Let $\chi_{\nu}$ denote the character of $\mathbf{M}(\K_{\nu})\mathbf{U}(\K_{\nu})$ given by the corresponding factor of $\chi_\varphi$ on $\mathbf{U}(\K_{\nu})$ and trivial on $\mathbf{M}(\K_{\nu})$.

\begin{theorem}[{\cite[Cor.~15.3]{GGP}, \cite{DSZ}}]\label{thm:local-uniquness}
For any irreducible admissible representation $\pi_\nu$ of $\mathbf{G}(\K_{\nu})$, we have $$\dim \Hom_{\mathbf{ M}(\K_{\nu})\mathbf{U}(\K_{\nu})}(\pi_\nu,\chi_{\nu})\leq 1$$
\end{theorem}

In fact, \cite{GGP,DSZ} prove a much stronger statement, that allows to tensor $\chi_{\nu}$ with any irreducible smooth admissible representation of $\mathbf{M}(\mathbb{K}_\nu)$.
 This stronger statement is called \emph{uniqueness of Bessel models}, but we will  not need it.

\begin{proof}[Proof of Corollary~\ref{cor:EuAD}]
  Let $\pi=\bigotimes_{\nu}\pi_{\nu}$ be the decomposition of $\pi$ into local products. By Theorem \ref{thm:Triv}, the restriction of $\cF_{S_{\alpha_1},\varphi}$ to $\pi$ is $\mathbf{M}(\K_{\nu})$-invariant. Thus, $\cF_{S_{\alpha_1},\varphi}$ induces for every $\nu$ an $\mathbf{M}(\K_{\nu})$-invariant functional on $\pi_{\nu}$. It is also of course $\mathbf{U}(\K_{\nu}),\chi$-equivariant. By Theorem \ref{thm:local-uniquness}, the space of such functionals is one-dimensional. Thus, $\cF_{S_{\alpha_1},\varphi}$ is Eulerian on $\pi$. By Theorem \ref{thm:transfer}, this implies that any Fourier--Jacobi coefficient $\mathcal{F}^I_{H, \varphi}$ is Eulerian on $\pi$.
\end{proof}

\section{Eisenstein series}
\label{sec:Eisenstein}

In this section, we study a special class of automorphic forms on $G = \mathbf{G}(\mathbb{A})$, namely Langlands--Eisenstein series $E(\lambda, g)$, depending on a complex parameter $\lambda$.  For special values of $\lambda$ such Eisenstein series belong to minimal or next-to-minimal representations of~$G$. We restrict to $\mathbb{K}=\mathbb{Q}$ throughout this section and let $\mathbb{A} = \mathbb{A}_\mathbb{Q}$.

We shall prove that Fourier--Jacobi coefficients of these Eisenstein series are Eulerian, thus proving Theorem~\ref{thm:ntm-Eisenstein}.
Using Theorem~\ref{thm:transfer}, this proceeds via a transfer of Eulerianity from degenerate Whittaker coefficients which are in turn shown to be Eulerian by computation using the reduction formula of Theorem~\ref{thm:reduction} below. Our results hold for the specific realizations of minimal and next-to-minimal representations given in Table~\ref{tab:SmallReps}.

\subsection{Eisenstein series and reduction formula for Whittaker coefficients}
Before proving all the next-to-minimal cases for ADE type of Table~\ref{tab:SmallReps}, we first have to introduce some notation for Eisenstein series, following \cite{FGKP}. We then recall the reduction theorem from \cite{FKP,FGKP} which is later used to show Eulerianity of Whittaker coefficients.

Let $\mathbf{G}$ be a semisimple algebraic group split and defined over $\mathbb{Q}$, and fix a Borel subgroup $\mathbf{B} = \mathbf{N} \mathbf{A}$ where $\mathbf{A}$ is a maximal torus and $\mathbf{N}$ is the unipotent radical.
Fix a maximal compact subgroup $K_\mathbb{A}$ of $\mathbf{G}(\mathbb{A})$ as in \cite[I.1.4]{MW95} such that
\begin{equation}
  \label{eq:Iwasawa}
  \mathbf{G}(\mathbb{A}) = \mathbf{N}(\mathbb{A}) \mathbf{A}(\mathbb{A}) K_\mathbb{A}
\end{equation}
and $K_\mathbb{A} = \prod_\nu K_\nu$ where $K_\nu$ is a maximal compact subgroup of $\mathbf{G}(\mathbb{Q}_\nu)$ and equals $\mathbf{G}(\mathfrak{o}_\nu)$ at almost all finite places which are in bijection with the rational prime numbers. The only infinite place for $\mathbb{K}=\mathbb{Q}$ is the archimedean real place $\nu=\infty$ and $\mathbf{G}(\mathbb{Q}_\infty)= \mathbf{G}(\mathbb{R})$ is a split real group of ADE type.

Let $X^*(\mathbf{A}) = \Hom(\mathbf{A}, \mathbb{G}_m)$ denote the lattice of rational characters of $\mathbf{A}$, where $\mathbb{G}_m$ is the multiplicative group, and let $\mathfrak{h}^* = X^*(\mathbf{A}) \otimes \mathbb{R}$ with vector space dual $\mathfrak{h}$.
Using the Iwasawa decomposition \eqref{eq:Iwasawa}, we define a logarithm map $H : \mathbf{G}(\mathbb{A}) \to \mathfrak{h}$ which is left-invariant under $\mathbf{N}(\mathbb{A})$, right-invariant under $K_\mathbb{A}$ and determined by its values on $\mathbf{A}(\mathbb{A})$ by
\begin{equation}
  e^{\langle \lambda| H(a) \rangle} = \abs{\lambda(a)}_\mathbb{A} \qquad \text{for all } \lambda \in \mathfrak{h}^*
\end{equation}
where $\langle \cdot | \cdot \rangle$ is the dual pairing and $\abs{\cdot}_\mathbb{A}$ is the adelic norm.
Let $\mathfrak{h}^*_\mathbb{C} = \mathfrak{h}^* \otimes_\reals \mathbb{C}$ and let $\rho \in \mathfrak{h}^*_\mathbb{C}$ denote the Weyl vector.
For each weight $\lambda \in \mathfrak{h}^*_\mathbb{C}$ we define a quasi-character $\tau_\lambda : \mathbf{B}(\mathbb{A})\to \mathbb{C}^{\times}$
by
\label{eq:chiB}
\begin{equation}
  \tau_\lambda(b) = e^{\langle \lambda + \rho |H(b) \rangle}.
\end{equation}
Note that $\tau_\lambda$ is Eulerian.

Consider now the principal series of smooth functions on $\mathbf{G}(\mathbb{A})$ defined by
\begin{equation}
\label{eq:indB}
I_{\mathbf{B}}(\lambda) \coloneqq \{f\, :\, \mathbf{G}(\mathbb{A})\to \mathbb{C}\, |\, f(bg)=\tau_\lambda(b)f(g), b\in \mathbf{B}(\mathbb{A})\}.
\end{equation}
The theory of Eisenstein series provides an automorphic realization of the principal series, parametrized by a choice of weight $\lambda$. To each $\lambda\in \mathfrak{h}_\mathbb{C}^*$ we define the following function on~$\mathbf{G}(\mathbb{A})$:
\begin{equation}
\label{eq:Eisdef}
E(\lambda, g)=\sum_{\gamma \in \mathbf{B}(\mathbb{Q})\backslash \mathbf{G}(\mathbb{Q})} e^{\left<\lambda + \rho |H(\gamma g)\right>}.
\end{equation}
This is invariant under the left action of $\mathbf{G}(\mathbb{Q})$ on $g$, while being right-invariant under $K_\mathbb{A}$. We will therefore refer to this as an \emph{Eisenstein series} which is the spherical vector of the automorphic realization of the induced representation~\eqref{eq:indB}.

For spherical Eisenstein series on split groups $\mathbf{G}$, the Whittaker coefficient~\eqref{WhittakerIntegral} with respect to a \textit{generic} character $\psi_N$ simplifies to the Eulerian integral
\begin{equation}
\label{eq:Jacquet}
\mathcal{W}_{\psi_N}[\lambda](g) \coloneqq \int_{\mathbf{N}(\mathbb{A})} \tau_{\lambda} (w_0ng)\psi_N(n)^{-1} dn.
\end{equation}
where $w_0$ is a representative of the longest word in the Weyl group $W$ of $\mathbf{G}$ and we have replaced $E(\lambda, g)$ in the functional argument by $\lambda$.
The local factors at finite places are called Jacquet integrals and can be computed by the Casselman--Shalika formula \cite{CasselmanShalika} related to characters of highest weight representations of the Langlands dual group.
For brevity let us henceforth denote the characters on $\mathbf{N}(\mathbb{A})$ simply by $\psi$. Note that for generic $\psi$ the integral~\eqref{eq:Jacquet} is only non-zero for generic $\lambda$.

When $\psi$ is degenerate the associated degenerate Whittaker coefficient is typically not Eulerian, unless one also chooses non-generic weights $\lambda$.
We first recall a reduction theorem for Whittaker coefficients for simply-laced groups $\mathbf{G}$ in the case of degenerate $\psi$~\cite{FKP}.  To this end we introduce some more notation.
Denote by $\Pi$ the set of simple roots of $\mathfrak{g}$.
As in \S\ref{sec:intro}, we choose a pinning of $\mathbf{G}$ from which we obtain a one-parameter subgroup $x_\alpha : \mathbb{G}_a \to \mathbf{G}$ for each simple root $\alpha$. Let $\chi$ be a fixed and unramified character on $\mathbb{A}$ trivial on $\mathbb{Q}$.
We may then represent the value of $\psi$ on any $n\in \mathbf{N}(\mathbb{A})$ as follows
\begin{equation}
  \label{eq:psi_N}
  \psi(n)=\psi\Biggl(\prod_{\alpha\in \Pi}x_\alpha(u_\alpha) n'\Biggr)=\chi\Biggl(\sum_{\alpha\in \Pi}m_\alpha u_\alpha\Biggr), \qquad m_\alpha\in \mathbb{Q}, u_\alpha \in \mathbb{A}
\end{equation}
where $n'\in [\mathbf{N},\mathbf{N}](\mathbb{A})$ and the value of $\psi$ does not depend on it.
We call the set of simple roots for which $m_\alpha\neq 0$ the \emph{support} of $\psi$ and denote it by $\text{supp}(\psi)$.
The character is generic if $\text{supp}(\psi)=\Pi$ and otherwise it is degenerate. If $\psi$ has support only on $k$ orthogonal simple roots we will say that it is \emph{degenerate of type} $k A_1$.

A degenerate character $\psi$ canonically defines a semi-simple proper subgroup $\mathbf{G}^{\prime}\subset \mathbf{G}$, such that the set of simple roots $\Pi^{\prime}$ of $\mathbf{G}^{\prime}$ equals the support of $\psi$. The Dynkin diagram of $\mathbf{G}^{\prime}$ is the subdiagram obtained from that of $\mathbf{G}$ by restricting to the nodes of $\Pi^{\prime}$. Let ${W}^{\prime}$ be the Weyl group of $\mathbf{G}^{\prime}$ with longest Weyl word $w_0^{\prime}$.
We then have the following:

\begin{theorem}\cite{FKP}
  \label{thm:reduction}
  Let $\mathbf{G}$ be a simple, simply-laced and split Lie group, defined over $\mathbb{Q}$ with a fixed maximal unipotent subgroup $\mathbf{N}$ and associated set of simple roots $\Pi$.
  Let $\psi$ be a character on $\mathbf{N}(\mathbb{A})$ and let $\mathbf{G}'$ be the subgroup of $\mathbf{G}$ determined by the set of simple roots $\Pi' \coloneqq \operatorname{supp}(\psi)$ as described above.
  Let $E(\lambda, g)$ be a spherical Eisenstein series on $\mathbf{G}(\mathbb{A})$ as defined in \eqref{eq:Eisdef}.
  Then the degenerate Whittaker coefficients $W_\psi^{\mathbf{G}}[\lambda](\id)$ of $E(\lambda, g)$ at the identity are determined by
\begin{align}
\label{eq:degW}
\mathcal{W}^\mathbf{G}_\psi[\lambda](\id) = \sum_{w_c w'_{\mathrm{long}}\in {W}/{W}'}
M(w_c^{-1},\lambda) \mathcal{W}^{\mathbf{G}'}_{\psi}[w_c^{-1}\lambda](\id)\,.
\end{align}
Here, $w_c$ is a coset representative of minimal length and can be constructed explicitly by Weyl orbit methods, and $\mathcal{W}_{\psi}^{\mathbf{G}'}[w_c^{-1}\lambda](\id)$ denotes a generic Whittaker coefficient of an Eisenstein series on the subgroup $\mathbf{G}'$ evaluated at the identity.
\end{theorem}

A similar formula for general arguments $g \in \mathbf{A}(\mathbb{A})$ can be found in~\cite{FGKP}.
The fact that we restrict to the identity $g = \id$ of $G$ to make later expressions shorter will not affect the argument for Eulerianity.
This is because we will show that, for the particular representations and coefficients we consider, there will be only one term in the reduction formula\eqref{eq:degW}.
Since each $\mathcal{W}^{\mathbf{G}'}_\psi$ in~\eqref{eq:degW} is a generic Whittaker coefficent, it is known to be Eulerian for general $g$ as described in \S\ref{sec:intro}.

\begin{remark}
The sum in the theorem above is over very specific coset representatives $w_c$.
The intertwiner $M(w_c^{-1},\lambda)$ is given by a quotient of completed Riemann zeta functions $\xi(s) = \pi^{-s/2} \Gamma(s/2) \zeta(s)$ through
\begin{align}
  \label{eq:intertwiner}
M(w_c^{-1},\lambda) = \prod_{\substack{\alpha>0\\w_c^{-1}\alpha<0}} \frac{\xi(\langle \alpha| \lambda\rangle)}{\xi(\langle \alpha| \lambda\rangle+1)}
\end{align}
where $\langle \cdot | \cdot \rangle$ here is the Killing form, normalized such that long roots satisfy $\langle \alpha | \alpha \rangle =2$.
\end{remark}

\begin{remark}
  If $\lambda$ corresponds to a quasi-character $\tau_\lambda$ on a standard parabolic subgroup $\widetilde{\mathbf{P}}=\widetilde{\mathbf{L}}\widetilde{\mathbf{U}}$, the Eisenstein series~\eqref{eq:Eisdef} becomes a parabolic Eisenstein series in a degenerate principal series. In this case the Weyl coset sum in~\eqref{eq:degW} reduces further to a double coset of Weyl groups $\widetilde{{W}} \backslash {W} / {W}'$ where $\widetilde{{W}}$ is the Weyl group of the Levi subgroup $\widetilde{\mathbf{L}}$, see (10.81) in~\cite{FGKP}. These double cosets typically have very few elements and the number of summands can be further reduced at special points in the degenerate principal series. These additional simplifications can have two origins and both depend on the choice of $\lambda$. First, the intertwiner \eqref{eq:intertwiner} can vanish at special points for the character $\lambda$, e.g.\ when one is studying a small representation as a residue of a degenerate principal series. Second, the Whittaker coefficient $\mathcal{W}^{\mathbf{G}'}_{\psi}[w_c^{-1}\lambda](\id)$ can vanish when $w_c^{-1}\lambda$ does not represent a generic quasi-character for $\mathbf{G}'$ as one is then trying to compute a generic Whittaker coefficient of a non-generic automorphic representation and this has to vanish.
\end{remark}

\subsection{Minimal and next-to-minimal representations}

A representation in the space of automorphic forms is said to have Gelfand--Kirillov dimension $d$ if all the orbits in its Whittaker support have dimension $2d$. It is conjectured that all irreducible automorphic representations have a Gelfand--Kirillov dimension, i.e. all the orbits in the Whittaker support have the same dimension \cite[Conjecture 5.10]{Ginz}. For automorphic realizations of a degenerate principal series $I_{\mathbf{P}}(\lambda)$ similar to~\eqref{eq:Eisdef}, it is conjectured that for a generic value of the parameter $\lambda$  the Gelfand--Kirillov dimension equals the dimension of the unipotent subgroup $\mathbf{U}$ of the  parabolic subgroup $\mathbf{P}=\mathbf{L}\mathbf{U}$ they are induced from.

For the purposes of this section we shall restrict to {\em maximal} parabolic subgroup $\mathbf{P}_{i_*}$ where $i_*$ denotes the simple root of $\mathbf{G}$ that is missing from the root system of the Levi subgroup $\mathbf{L}_{i_*}$. The corresponding Eisenstein series of the automorphic realization of the degenerate principal series can be obtained by using the specific weight
\begin{equation}
\label{eq:lambdais}
\lambda_{i_*}\!(s)=2s\Lambda_{i_*}-\rho, \qquad s\in \mathbb{C}\,,
\end{equation}
where $\Lambda_{i_*}$ denotes a fundamental weight that satisfies $\langle \Lambda_{i_*} | \alpha_j\rangle = \delta_{i_*,j}$ with the simple roots $\alpha_j$.

\begin{table}[t]
  \centering
  \caption{Eisenstein series realizations of minimal and next-to-minimal representations. The notation $s_{i_*}$ here labels at the same time the maximal parabolic $\mathbf{P}_{i_*}$ and the value of the parameter $s$ appearing in~\eqref{eq:lambdais}. For type $\SO_{n,n}$ and $n\geq 5$ there are two distinct next-to-minimal representations that we realize automorphically.  The non-generic different choices give isomorphic automorphic representations due to Langlands's functional relations for Eisenstein series.
  }
  \label{tab:SmallReps}
  \begin{tabular}{ccc}
    \toprule
    Lie group & $\pi_\text{min}$ & $\pi_\text{ntm}$ \\
    \midrule
    $\SL_n$ & generic $s_1$ \textbf{or} generic $s_{n-1}$ & generic $s_2$ \textbf{or} generic $s_{n-2}$ \\[0.75em]
    $\SO_{n,n}$ & $s_1 = \tfrac{n-2}{2}$ \textbf{or} $s_n = 1$ \textbf{or} $s_{n-1}$ = 1 &
    $\begin{cases} \text{generic $s_1$} & (2A_1)'\\[0.1em] s_{n-1}=2\,\, \textbf{or}\,\, s_n=2 & (2A_1)''\end{cases}$      \\[1.5em]
    $\E_{6(6)}$ & $s_1 = \tfrac{3}{2}$ \textbf{or} $s_6 = \tfrac{3}{2}$ & generic $s_1$ \textbf{or}  generic $s_6$ \textbf{or} $s_5 = 1$ \\[0.5em]
    $\E_{7(7)}$ & $s_1 = \tfrac{3}{2}$ \textbf{or} $s_7 = 2$ & $s_1 = \frac{5}{2}$ \textbf{or} $s_6 = \tfrac{3}{2}$ \textbf{or} $s_7 = 4$ \\[0.5em]
    $\E_{8(8)}$ & $s_1 = \frac{3}{2}$ \textbf{or} $s_8 = \tfrac{5}{2}$ & $s_1 = \frac{5}{2}$ \textbf{or} $s_7 = 2$ \textbf{or} $s_8 = \tfrac{9}{2}$ \\[0.5em]
    \bottomrule
  \end{tabular}
\end{table}

The Gelfand--Kirillov dimension of the associated automorphic representation is conjectured in \cite[\S 5]{Ginz} to satisfy
\begin{align}
  \operatorname{GKdim}\big(I_{\mathbf{P}_{i_*}}(\lambda_{i_*}(s))\big) = \dim \mathbf{U}_{i_*}\quad \text{for all } s\in \mathbb{C}.
\end{align}

However, for specific choices of $s$ there can be  spherical submodules/quotients with smaller GK-dimension; this is one of the standard ways of realising minimal and next-to-minimal automorphic representations~\cite{GRS2,KazhdanSavin,GanSavin,GMV,FGKP}. In Table~\ref{tab:SmallReps}, we list particular ways of realizing minimal and next-to-minimal representations of groups of ADE-type based on these references. More specifically, the minimal and next-to-minimal cases for type $A_n$ as full inductions from maximal parabolics are well known and we verify them using the methods of this section below. 
The minimal cases of $D_n$, $E_6$, $E_7$ and $E_8$ in the Heisenberg parabolic were treated in~\cite{GRS2}, the next-to-minimal case for $E_6$, $E_7$ and $E_8$ in the parabolic $\mathbf{P}_1$ were given in~\cite{GMV} and the other realizations given in the table are related by functional relations of Eisenstein series. The next-to-minimal cases for $D_n$ can be determined using the methods of this section and we shall present the details for spinor nodes below where we also derive the Eulerianity of the next-to-minimal coefficients. We analyze also the exceptional cases using this method and derive the next-to-minimal Whittaker support of the Eisenstein series at $s_8=\frac92$ in detail.
The realizations in Table~\ref{tab:SmallReps} were chosen because they are of interest in string theory where the extremal nodes of the Dynkin diagram have an interpretation related to certain weak-coupling limits of the theory \cite{FGKP}.
At the values shown in the table,  minimal and next-to-minimal representations appear as subrepresentations.
The corresponding Gelfand--Kirillov dimensions are tabulated in Table~\ref{tab:GKdim}. Note that for type D there are two distinct next-to-minimal orbits of type $2A_1$ (see \S\ref{sec:Bessel}) and they are listed separately. We follow Bourbaki labelling conventions.

\begin{table}[t]
  \centering
  \caption{Gelfand--Kirillov dimensions for minimal and next-to-minimal representations described in Table~\ref{tab:SmallReps}. The Whittaker support of the corresponding Eisenstein series is shown to be given by the nilpotent orbits whose dimension is listed in the table.
  For $\SL_3$ there exists no next-to-minimal representation whose Whittaker support is a $2A_1$-orbit since there is no such orbit. For $\SO_{n,n}$ and $n\geq 5$ there are two $2A_1$ orbits that are listed separately.}
  \label{tab:GKdim}
  \begin{tabular}{ccl}
    \toprule
    Lie group & $\operatorname{GKdim}(\pi_\text{min})$ & $\operatorname{GKdim}(\pi_\text{ntm})$ \\
    \midrule
    $\SL_n$ & $n-1$ & $2n-4$ \\[0.25em]
    $\SO_{n,n}$ & $2n-3$ & $\hspace{-0.8em}\begin{cases} 2n-2 & (2A_1)' \\ 4n-10 & (2A_1)''\end{cases}$\\[1.25em]
    $\E_{6(6)}$ & $11$ & $16$ \\
    $\E_{7(7)}$ & $17$ & $26$ \\
    $\E_{8(8)}$ & $29$ & $46$ \\
    \bottomrule
  \end{tabular}
\end{table}

\subsection{Eulerianity for next-to-minimal representations of ADE-type}

In this section, we give explicit examples of degenerate Whittaker coefficients for the minimal and next-to-minimal representations of Table~\ref{tab:SmallReps}.

For convenience we will introduce the following labelling of the coefficients and their characters.
As shown in \eqref{eq:psi_N}, a character $\psi$ on $N$ can be parametrized by a tuple $[m_{\alpha_1}, \ldots, m_{\alpha_r}]$ with $m_{\alpha_i} \in \mathbb{Q}$.
We will denote the corresponding Whittaker coefficient of an Eisenstein series $E(\lambda, g)$ as $\mathcal{W}_{[m_{\alpha_1}, \ldots, m_{\alpha_r}]}(g)$ suppressing the dependence of $\lambda$.
We shall also evaluate all Whittaker coefficients at the identity $\id$ of $\mathbf{G}(\mathbb{A})$ to make the resulting expressions simpler.

For $m \in \mathbb{Z}$ and $s \in \mathbb{C}$ we shall also use the notation
\begin{align}
\label{eq:SL2W}
B_{m}(s) \coloneqq  \frac{2}{\xi(2s)} |m|^{s-1/2} \sigma_{1-2s}(m) K_{s-1/2} (2\pi |m|)
\end{align}
for the standard $\SL_2$ Whittaker coefficient evaluated at the identity with $\sigma_{1-2s}(m)$ being the divisor sum and $K_{s-1/2}$ the modified Bessel function of the second kind. We note that this function does not vanish unless $s=0$, corresponding to the value where the non-holomorphic $\SL_2$ Eisenstein series belongs to the trivial representation (and thus has vanishing Fourier coefficients), or when $s=1/2$, corresponding to the value where there is a zero in the degenerate principal series. Moreover, the expression~\eqref{eq:SL2W} is always Eulerian.

Similarly, we shall use the notation
\begin{align}
\label{eq:SL3W}
B_{m_1,m_2}(s_1,s_2)\coloneqq \frac{1}{\xi(2s_1)\xi(2s_2)\xi(2s_1+2s_2-1)} \tilde{K}_{s_1,s_2}(m_1,m_2)\,,
\end{align}
for the generic $\SL_3$ Whittaker coefficient evaluated at the identity. The function $\tilde{K}_{s_1,s_2}(m_1,m_2)$ is a special function that is given by a convolution of two $K$-Bessel functions~\cite{Bump:1984,Pioline:2009}. Its precise form does not matter to us here since it does not vanish for any value of $s_1$ and $s_2$. The vanishing is solely controlled by the prefactor. As~\eqref{eq:SL3W} corresponds to a generic $\SL_3$ Whittaker coefficient, it is Eulerian for generic (non-vanishing) $m_1$ and $m_2$.

Our focus will be mainly on type $A_1$ and type $2A_1$ Whittaker coefficients that are characterized by a single simple root or a pair of orthogonal simple roots, respectively. All single simple roots are Weyl-conjugate. According to~\cite[Cor.~3.0.3]{NTMSimplyLaced} all orthogonal pairs are conjugate for simple groups of types $A$ and $E$ while for $D_n$ with $n>4$ there are two inequivalent pairs, corresponding to the two types of $2A_1$ orbits and next-to-minimal representations listed in Table~\ref{tab:GKdim}. We shall only analyse one inequivalent pair of orthogonal for each next-to-minimal representation as other choices are conjugate.

\subsubsection{Type \texorpdfstring{$A_n$}{An}}

For type $A_n$, we consider first minimal representations. These come in a one-parameter family that corresponds to the degenerate principal series induced from $\mathbf{P}_1$ (or $\mathbf{P}_n$). To show Eulerianity of rank-one Fourier coefficients it suffices to compute a single Whittaker coefficient that we choose to be supported on the simple root $\alpha_1$. Applying~\eqref{eq:degW} we obtain for $\lambda = 2s \Lambda_1 -\rho$
\begin{align}
\mathcal{W}_{[m,0,0,\ldots,0]} (\id) =  B_m(s)
\end{align}
and is clearly Eulerian.  We have also checked that Whittaker coefficients supported on more than a single root vanish by using Theorem~\ref{thm:reduction}.

For the next-to-minimal representation given in Table~\ref{tab:SmallReps} and $n>2$ we consider the parabolic $\mathbf{P}_2$ according to Table~\ref{tab:SmallReps}, i.e. $\lambda=2s\Lambda_2-\rho$. Choosing the degenerate character to be supported on nodes $1$ and $3$, one finds
\begin{align}
\mathcal{W}_{[m,0,n,\ldots,0]} (\id) =  \frac{\xi(2s-1)}{\xi(2s)}B_m(s-\tfrac12)B_n(s-\tfrac12)\,,
\end{align}
which is an Eulerian expression. Eulerianity of maximal parabolic Fourier coefficients then follows from Theorem~\ref{thm:transfer} or from~\cite[Thm.~B(iii)]{Ahlen:2017agd}. Whittaker coefficients associated with larger orbits can be shown to vanish by using Theorem~\ref{thm:reduction}.

\subsubsection{Type \texorpdfstring{$D_n$}{Dn}} According to Table~\ref{tab:SmallReps}, the minimal representation of $D_n$ (for $n\geq 4$) can be realized in the degenerate principal series induced $I_{\mathbf{P}_1}$ for the value $s=1$. The corresponding Whittaker coefficient is determined from~\eqref{eq:degW} to be
\begin{align}
\mathcal{W}_{[m,0,0,\ldots,0]} (\id) =  B_m\left(\tfrac{n-2}{2}\right)
\end{align}
and is Eulerian. One can check that there are no non-zero Whittaker coefficients associated with larger orbits. 

Let us deduce that the representation $\pi$ generated by the Eisenstein series $E(\lambda_1(1),g)$, where we use the notation~\eqref{eq:lambdais},  is indeed minimal. Since $\pi$ is a quotient of the automorphic realization of the degenerate principal series $I_{\mathbf{P}_1}$, the local components are quotients of local degenerate principal series. 
Thus, every  orbit of the Whittaker support of $\pi$ lies in the closure of the Whittaker support of the local degenerate principal series for each place. The behavior of the local Whittaker support under induction is determined for $p$-adic places in \cite{MW}, and for degenerate principal series the Whittaker support lies in the Richardson orbit corresponding to the parabolic subgroup. 

The (complex) Richardson orbit of the parabolic $\mathbf{P}_1$ is given by the partition $(2^n)$ for even $n$ and by the partition $(2^{n-1}1^2)$ for odd $n$. Each complex orbit in the closure of this one is given by a partition of the form $(2^k1^{2(n-k)})$, and every such orbit includes a unique rational orbit. Moreover, all these rational orbits allow Whittaker coefficients. Lemma \ref{lem:neutral-dominates} and Theorem \ref{thm:IntTrans} imply now that for every orbit in the Whittaker support of the Eisenstein series $E(\lambda_1(1),g)$, the corresponding Whittaker coefficient does not vanish. Since $E(\lambda_1(1),g)$ has no non-zero Whittaker coefficients associated with non-zero non-minimal orbits, we deduce that the only orbit in its Whittaker support is the minimal orbit.

For $D_n$ ($n>4$), there are two different next-to-minimal representations to discuss from Table~\ref{tab:SmallReps} and we consider both in turn.

We begin with the degenerate principal series for $D_n$ associated with the maximal parabolic $P_1$ in Bourbaki enumeration. This is a next-to-minimal representation of Gelfand--Kirillov dimension $2n-2$, corresponding to the partition $(31^{n-3})$ and orbit $(2A_1)'$, for generic members of the degenerate principal series. Among the pairs of orthogonal simple roots, only the pair $(\alpha_{n-1},\alpha_n)$ of spinor nodes belongs to the orbit $(2A_1)'$~\cite{NTMSimplyLaced} and we shall now evaluate the associated degenerate Whittaker vector using formula~\eqref{eq:degW}. For this one needs to study the double cosets $W(A_1A_1)\backslash W(D_n) / W(D_{n-1})$ and there are $2n-3$ of them. Inspection shows that there is only one element with a non-vanishing Whittaker coefficient and this becomes Eulerian.

As an example, for $D_5$ and the Eisenstein series with weight $\lambda=2s\Lambda_1-\rho$ one obtains the Eulerian Whittaker coefficient (at the identity)
\begin{align}
\mathcal{W}_{[0,0,0,m,n]} (\id) =
\frac{\xi (2 s-4)^2 }{\xi (2 s) \xi (2 s-3)} B_{m}(\tfrac{5}{2}-s)B_n(\tfrac{5}{2}-s)\,,
\end{align}
while for $D_6$ the similar Whittaker coefficient is given by
\begin{align}
\mathcal{W}_{[0,0,0,0,m,n]} (\id) =
\frac{\xi (2 s-5)^2
}{\xi (2 s) \xi (2 s-4)}
B_m(3-s)B_n(3-s)\,.
\end{align}
The Eulerianity of these coefficients can be transferred to other Fourier coefficients using Theorem~\ref{thm:transfer}.

For $D_n$ and $n>4$ one can also generate next-to-minimal representations with Whittaker support given by the other next-to-minimal orbit $(2A_1)''$ from degenerate principal series associated with the `spinor' nodes $\alpha_n$ or $\alpha_{n-1}$. We shall choose $\alpha_n$ for concreteness in the following. The Whittaker support of generic members of this degenerate principal series of $D_n$ has Bala--Carter label $\lfloor \tfrac{n}{2}\rfloor A_1$
and thus exceeds the next-to-minimal type $2A_1$ for $n\geq 6$. Therefore one has to consider special members of the degenerate principal series, i.e., special values of $s$. Before addressing the general case, we illustrate the idea in the case of $D_6$.

We shall study the degenerate principal series $I_{\mathbf{P}_6}$ of $D_6$ and first look at generic values of $s$ where the representation is larger than next-to-minimal and then show how for the value $s=2$ of Table~\ref{tab:SmallReps} the automorphic representation reduces to a next-to-minimal representation.
For generic $s$, the Whittaker support is of Bala--Carter type $3A_1$ as can be seen by studying the degenerate Whittaker coefficients.
For an Eisenstein series in this degenerate principal series, a Whittaker coefficient of type $3A_1$ is then determined according to~\eqref{eq:degW} as
\begin{align}
\label{eq:3A1D6}
\mathcal{W}_{[m,0,n,0,0,p]} (\id) =
\frac{\xi (2 s-5)^3
}{\xi (2 s) \xi (2 s-4)\xi (2 s-2)} B_{m} (3-s)B_n(3-s)B_p(3-s)\,.
\end{align}
This is Eulerian as one expects for a character in a maximal orbit of the wave-front set.

By comparison, the following Whittaker coefficient of type $2A_1$ is, in the generic degenerate principal series $I_{\mathbf{P}_6}$, given by
\begin{align}
\mathcal{W}_{[m,0,n,0,0,0]} (\id) &=
\frac{\xi (2 s-5)^2
}{\xi (2 s) \xi (2 s-2)} B_{m} (3-s) B_{n} (3-s)\nonumber\\
&\quad+\frac{\xi (2 s-5)^3
}{\xi (2 s) \xi (2 s-4) \xi (2 s-2)} B_{m} (3-s) B_{n} (3-s)
\end{align}
and is clearly not Eulerian.
For the special value $s=2$, however, the generic spinor reduces to a next-to-minimal representation with Eulerian Whittaker coefficient since the second term vanishes. Moreover, the $3A_1$ Whittaker coefficient~\eqref{eq:3A1D6} vanishes for $s=2$ and therefore the Whittaker support of the automorphic representation is reduced to $(2A_1)''$ at this value of $s$.

For generic $D_n$ with $n>6$, the Whittaker support of the degenerate principal series induced from $\mathbf{P}_n$ is  $\lfloor \tfrac{n}{2}\rfloor A_1$.
Computing all the degenerate Whittaker coefficients in this case is not very feasible and instead we analyse the Fourier coefficients of the $\mathbf{P}_n$ degenerate principal series with respect to the unipotent radical of the parabolic subgroup $\mathbf{P}_1$. The complete Fourier expansion was already obtained in~\cite[(B.8)]{Bossard:2015oxa}  and  can be written as
\begin{align}
\label{genDspinor}
 \xi(2s)  E^{D_{n}}(\lambda_n(s))&= 2 \xi(2s)
 E^{D_{n-1}}(\lambda_{n-1}(s))  + 2 \xi(2s-n)
 E^{D_{n-1}}(\lambda_{n-2}(s-1))\nonumber \\
&\hspace{5mm} +
\hspace{-3mm}
\sum_{\substack{Q \in \mathbb{Z}^{2(n-1)}_\times\\ \langle Q,Q\rangle = 0 }}
f_Q(s,n)
E^{D_{n-2}}(\lambda_{n-3}(s-1))
e^{2\pi i \langle Q, u\rangle }\ ,
\end{align}
where $Q\in \mathbb{Z}^{2(n-1)}_\times$ labels a non-vanishing character on the unipotent $\mathbf{U}_1$ of $D_n$ and $u\in \mathfrak{u}_1$ is a Lie algebra element so that $e^{2\pi i \langle Q, u \rangle }$ is the corresponding Fourier mode. The constraint on the sum means that the vector $Q$ is null in the split signature lattice and this characterizes  minimal elements in the character variety $\mathfrak{u}_1^*$.

The first line represents the constant term in this unipotent Fourier expansion and is given by maximal parabolic Eisenstein series on the semi-simple Levi $D_{n-1}$. The non-trivial Fourier coefficients in the second line are only non-zero for elements in the minimal $D_{n-1}$-orbit in $\mathfrak{u}_1^*$ that is given by non-vanishing null vectors $Q$. These minimal elements are the intersection of the minimal $A_1$-type of orbit of $D_n$ with $\mathfrak{u}_1^*$. The corresponding Fourier coefficient is composed out of a non-vanishing Eulerian function $f_Q(s,n)$, that is explicitly known, and an Eisenstein series on $D_{n-2}$. Here, $D_{n-2}$ arizes as the stabilizer (in $D_{n-1}$) of the null vector $Q$ and is obtained by deleting the first two nodes of the original $D_n$ diagram. The function $f_Q(s,n)$ is given by a product of a certain divisor sum (contribution from the finite places) and a modified Bessel function (archimedean contribution), and has a trivial Fourier expansion on the Levi $D_{n-1}$.
Using~\eqref{genDspinor} one can now identify the minimal and next-to-minimal representations in the degenerate principal series $I_{\mathbf{P}_n}$ of $D_n$.

The value $s=1$ always realizes a minimal representation. This can be seen in the above formula as follows: For $s=1$ the Eisenstein series on the Levi subgroups $D_{n-1}$ and $D_{n-2}$ in the second and third terms belong to the trivial representation and are equal to one. The sum over $Q$ in the third term contains only minimal characters in the Fourier mode  $e^{2\pi i \langle Q, u \rangle }$ and so is of type $A_1$ as is thus the full second line since $E^{D_{n-2}}$ is trivial. The first term is by induction minimal at $s=1$ and the second term belongs to the trivial orbit. Therefore, $s=1$ is by induction an automorphic realization of a minimal representation, in agreement with Table~\ref{tab:SmallReps}. Moreover, the Fourier coefficient is Eulerian as can be checked by induction or using the fact that this automorphic realization of a minimal representation is related functionally to the one at the beginning of the $D_n$ section which was checked to have an Eulerian Whittaker coefficient.

A next-to-minimal representation is always realized at $s=2$.
 This can be seen in~\eqref{genDspinor} as follows. The first term in the first line belongs to a next-to-minimal representation by induction. The second term belongs to the minimal orbit as its proper $s$-value is then $s-1=1$ which was analysed above to be minimal. The last term has an Eisenstein series on the stabilizer subgroup $D_{n-2}$ at $s-1=1$ which again is minimal multiplied with a minimal character on the character variety, so together they are always of type $2A_1$ and hence next-to-minimal. The two $A_1$s are orthogonal since one is associated with node $1$ of the original $D_n$ diagram and the other with any of the nodes of $D_{n-2}$ that is disconnected from node $1$.  The $2A_1$ Fourier modes in the Whittaker support of the next-to-minimal representation at $s=2$ are moreover Eulerian as the function $f_Q(s,n)$ is can be seen by inspection of its explicit form in~\cite{Bossard:2015oxa}.

\subsubsection{Type \texorpdfstring{$E_6$}{E6}}

The automorphic realization of the minimal representation given in Table~\ref{tab:SmallReps} by $s_1=3/2$ and the associated $A_1$-type Whittaker coefficient were given in~\cite{Dvorsky1999,Kazhdan:2001nx,KazhdanPolishchuk,Savin2007,FKP} and shown to be Eulerian.

There is a one-parameter family of next-to-minimal representations for $E_6$, see Table~\ref{tab:SmallReps}.
A next-to-minimal Eisenstein series can be obtained for any value of $s$ in the degenerate principal series induced from the parabolic $\mathbf{P}_1$ or $\mathbf{P}_6$. Focusing on $\mathbf{P}_1$ for concreteness, this family can be parametrised by taking a generic character $\lambda=2s\Lambda_1-\rho$ with $s\in\mathbb{C}$. We shall now analyse the degenerate Whittaker vectors of type $A_1$ and type $2A_1$.

Consider first taking a minimal character $\psi$ such that $\mathbf{G}'$ is of type $A_1$ and take this for simplicity to be the node that also defines $\widetilde{\mathbf{G}}$, i.e., we take $\mathbf{G}'$ to correspond to node $1$. Then there are $21$ double coset elements in the sum of \eqref{eq:degW}. For all but $6$ of these the Whittaker coefficient $\mathcal{W}^{\mathbf{G}'}_{\psi}$ vanishes since $w_c^{-1}\lambda$ is not generic on that $A_1$. Thus we have a simplification but still a non-Eulerian degenerate Whittaker coefficient. There are non-generic choices for $s$ (see Table~\ref{tab:SmallReps}) where there are simplifications coming from the intertwiner and these correspond to embedding the minimal representation in this degenerate principal series and at these special points one is left with a single term in the sum, leading to an Eulerian Whittaker coefficient.

Now consider a next-to-minimal character $\psi$ with $\mathbf{G}'$ of type $2A_1$ that we take to be defined on the orthogonal nodes $1$ and $4$.
Evaluating the resulting Whittaker coefficient $\mathcal W^{\mathbf{G}}_\psi$ in the degenerate principal series all but one term in the sum disappear with the result
\begin{align}
\mathcal{W}_{[m,0,0,n,0,0]} (\id) = \frac{\xi (2 s-7)^2 }{\xi (2 s) \xi (2 s-3)} B_{m}(4-s) B_{n}(4-s)\,.
\end{align}
This represents an Eulerian expression for the type $2A_1$ Whittaker coefficient in the next-to-minimal representation. By Theorem~\ref{thm:transfer} this Eulerianity transfers to other unipotent Fourier coefficients.

\subsubsection{Type \texorpdfstring{$E_7$}{E7}}

The automorphic realization of the minimal representation given in Table~\ref{tab:SmallReps} by $s_1=3/2$ and the associated $A_1$-type Whittaker coefficient were given in~\cite{Dvorsky1999,Kazhdan:2001nx,KazhdanPolishchuk,Savin2007,FKP} and shown to be Eulerian.

For $E_7$, we consider the so-called abelian realisation which means the degenerate principal series associated with the maximal parabolic $\mathbf{P}_7$ where node $7$ is singled out. For generic members of this degenerate principal series, the Whittaker support is $3A_1$~\cite{MillerSahi}, but for special points this reduces to $2A_1$ and gives an automorphic realization of a next-to-minimal representation, see Table~\ref{tab:SmallReps}.

We focus immediately on the $2A_1$ Whittaker coefficient and choose a character $\psi$ that is supported on nodes $1$ and $7$.
For generic $s$ one has the non-Eulerian expression
\begin{align}
\mathcal{W}_{[m,0,0,0,0,0,p]} (\id) &=
\frac{\xi (2 s-9) \xi (2 s-6)
}{\xi (2 s) \xi (2 s-4)}B_{m}(\tfrac{13}{2}-s)B_{n}(5-s)\nonumber\\
&\quad+\frac{\xi (2(s-6)) \xi (2 s-9)^2
}{\xi (2 s) \xi (2s-8) \xi (2 s-4)}B_{m}(\tfrac{7}{2}-s)B_{n}(5-s)\,.
\end{align}
For $s=4$ the second term vanishes and one is left with a single Eulerian coefficient as expected for an orbit in the Whittaker support of a next-to-minimal representation. By Theorem~\ref{thm:transfer} this Eulerianity can be transferred to other Fourier--Jacobi coefficients. One can also check that Whittaker coefficients for larger orbits vanish.

\subsubsection{Type \texorpdfstring{$E_8$}{E8}}

The automorphic realization of the minimal representation given in Table~\ref{tab:SmallReps} by $s_1=3/2$ and the associated $A_1$-type Whittaker coefficient were given in~\cite{Dvorsky1999,Kazhdan:2001nx,KazhdanPolishchuk,Savin2007,FKP} and shown to be Eulerian.

For $E_8$, we consider the degenerate principal series in the Heisenberg realisation which is associated with the maximal parabolic subgroup $\mathbf{P}_8$.
The Whittaker support for generic members of the principal degenerate series is type $A_2$ (lying over the non-special $3A_1$ in the $E_8$ Hasse diagram~\cite{Spaltenstein}) and this has to reduce to $2A_1$ for the correct $s$-value given as $s=9/2$ in Table~\ref{tab:SmallReps}.

We first determine a Whittaker coefficient of type $A_2$ for generic members of the degenerate principal series
\begin{align}
\mathcal{W}_{[0,0,0,0,0,0,m,n]} (\id) &=
\frac{\xi (2 (s-9)) \xi (2 (s-7)) \xi (2 s-11) \xi (4 s-29) }{\xi (4 (s-7)) \xi (2 s) \xi (2 s-9) \xi (2 s-5)}
B_{m,n}(6-s,\tfrac{19}{2}-s)
\end{align}
where $B_{m_1,m_2}(s_1,s_2)$ was defined in~\eqref{eq:SL3W}. This Whittaker coefficient is Eulerian but it vanishes when $s=9/2$ which is the value when the degenerate principal series has a subrepresentation corresponding to a next-to-minimal representation.

In this case it is also interesting to consider the $3A_1$ coefficient. For generic $s$ one instance is
\begin{align}
\mathcal{W}_{[0,0,0,m,0,n,0,p]} (\id) &=
\frac{\xi (2 s-11)^3
}{\xi (2 s) \xi (2 s-9) \xi (2 s-5)}
B_{m}(6-s) B_{n}(6-s) B_{p}(6-s)\nonumber\\
&\quad +\frac{\xi (2 (s-9))^3 \xi (4 s-29)}{\xi (4 (s-7)) \xi (2 s) \xi (2 s-9) \xi (2 s-5)} B_{m}(\tfrac{19}{2}-s)B_{n}(\tfrac{19}{2}-s)B_{p}(\tfrac{19}{2}-s)
\end{align}
and it is clearly not Eulerian. However, we see that it vanishes completely for $s=9/2$ which is to be expected since the orbit $3A_1$ is not special and therefore should not be the Whittaker support of any automorphic representation. This is the phenomenon of `raising nilpotent orbits'~\cite{JLS}.

Let us finally consider a $2A_1$ Whittaker coefficient and the next-to-minimal representation. We choose the character $\psi$ for on nodes $6$ and $8$ and formula~\eqref{eq:degW} yields
\begin{equation}
\begin{split}
  \MoveEqLeft
\mathcal{W}_{[0,0,0,0,0,m,0,n]} (\id) = \tfrac{ \xi (2 s-11)^2}{\xi (2 s) \xi   (2s-5)}
       B_{m}(6-s)B_n(6-s)
      \\
& +\tfrac{ \xi (2 (s-9)) \xi (2 (s-7)) \xi (2 (s-6))\xi (4 s-29)}{\xi (4 (s-7)) \xi (2 s) \xi (2 s-9) \xi (2s-5)}
    B_{m}(\tfrac{13}{2}-s) B_{n}(\tfrac{19}{2}-s)
   \\
&+\tfrac{ \xi (2 s-11)^2 \bigl(  \xi (2 s-11)
+ \xi (2 s-13)
+ \xi (2 (s-6))
+\xi (2 (s-5)) \bigr) }{\xi (2 s) \xi(2 s-9) \xi (2 s-5)}
        B_{m}(6-s)B_n(6-s)
      \\
&+\tfrac{ \xi (2 (s-9))^2 \xi (4 s-29) \bigl( \xi (2 s-19) + \xi (2 s-17) + \xi (2 s-15) + \xi (2 (s-9)) + \xi (2 (s-8)) \bigr)}{\xi (4 (s-7)) \xi (2s) \xi (2 s-9) \xi (2s-5)}
   B_{m}(\tfrac{19}{2}-s)B_n(\tfrac{19}{2}-s)
      \\
&+\tfrac{ \xi (2 (s-7)) \xi (2 s-17) \xi (2s-11) }{\xi (2 s) \xi (2 s-9) \xi (2 s-5)}
          B_{m}(9-s)B_n(6-s)\,.
\end{split}
\end{equation}
For the special value $s=9/2$ all terms but the first one vanish and one is left with an Eulerian expression. This shows that the $2A_1$ Whittaker coefficients become Eulerian in the next-to-minimal representation of Table~\ref{tab:SmallReps}.

Theorem~\ref{thm:transfer} then allows to conclude Eulerianity for all Fourier--Jacobi coefficients with characters in the above orbits, thus proving Theorem~\ref{thm:ntm-Eisenstein}.

{\small
\newcommand{\etalchar}[1]{$^{#1}$}
\def\cprime{$'$}
\providecommand{\href}[2]{#2}\begingroup\raggedright\endgroup
}

\end{document}